\documentclass[12pt]{amsart}

\usepackage{mathrsfs}
\newtheorem{theorem}{Theorem}[section]

\newtheorem{prop}[theorem]{Proposition}
\newtheorem{cor}[theorem]{Corollary}
\newtheorem{definition}[theorem]{Definition}
\newtheorem{remark}[theorem]{Remark}
\newtheorem*{theorem*}{Theorem}
\def\bl{\boldsymbol}

\def\mob{\text{M\"ob}}

\theoremstyle{definition}

\usepackage{soul}
\setstcolor{red}
\usepackage{marginnote}
\usepackage[pagebackref,hypertexnames=false, colorlinks, citecolor=red,linkcolor=blue, urlcolor=red]{hyperref}

\theoremstyle{remark}

\numberwithin{equation}{section}

\newcommand{\overbar}[1]{\mkern 1.5mu\overline{\mkern-1.5mu#1\mkern-1.5mu}\mkern 1.5mu}

\usepackage{lmodern}

\usepackage{enumitem}

\usepackage{color}
\usepackage{comment}
\usepackage{fullpage}

\usepackage{ulem}

\usepackage{tikz-cd}

\begin{document}

% \title[short text for running head]{full title}
% \title[Homogeneous pairs of multiplication operators -- automorphism group of symmetrized bi-disc]{Pairs of multiplication operators homogeneous under the automorphism group of the symmetrized bi-disc}

\title[M\"ob-homogeneity]{M\"OB-HOMOGENEOUS ANALYTIC HILBERT MODULES OVER THE BIDISC}

%    Only \author and \address are required; other information is
%    optional.  Remove any unused author tags.

%    author one information
% \author[short version for running head]{name for top of paper}
\author[J. Das]{Jyotirmay Das}
\address[J. Das]{Institute for Advancing Intelligence\\ TCG Centres for Research and Education in Science and Technology \\ Kolkata - 700091, and NIT Durgapur\\ India}
\curraddr{}
\email{iamjyotirmay1999@gmail.com}
\thanks{The first-named author gratefully acknowledges the generous support provided by TCG CREST Ph.D Fellowships.}

%    author two information
\author[S. Hazra]{Somnath Hazra}
\address[S. Hazra]{Institute for Advancing Intelligence\\ TCG Centres for Research and Education in Science and Technology \\ Kolkata - 700091}
\email{somnath.hazra@tcgcrest.org}
\thanks{}

%\author[G. Misra]{Gadadhar Misra}
%\address[G. Misra]{Indian Statistical Institute, Bangalore - 560059\\           and Indian Institute of Technology, Gandhinagar - 382055, India}
%\curraddr{}
%\email{gm@isibang.ac.in}

%    \subjclass is required.
\subjclass[2020]{Primary 47A13, 47B13, 47B32; Secondary 46E22}
\keywords{Analytic Hilbert Modules, M\"ob-homogeneous, Curvature, Reproducing kernel, K\"ahler-Einstein metric}
%\date{\today}

\dedicatory{}

%    "Communicated by" -- provide editor's name; required.
%\commby{}

\begin{abstract} An analytic Hilbert module $\mathcal{H}$ over the polynomial ring, consisting of holomorphic functions over the bidisc, is said to be M\"ob-homogeneous if the corresponding pair of multiplication operators is homogeneous with respect to the diagonal action of the group $\{(\varphi,\varphi): \varphi \in \mbox{M\"ob}\} \cong \mbox{M\"ob}$. In this article, we construct three families of mutually unitarily inequivalent M\"ob-homogeneous analytic Hilbert modules, distinct from the family of weighted Bergman modules over the bidisc. We further show that none of the reproducing kernels in one of these families induces a K\"ahler--Einstein metric on the bidisc.
\end{abstract}

\maketitle

%    Text of article.

% Bibliographies can be prepared with BibTeX using amsplain,
%    amsalpha, or (for "historical" overviews) natbib style.
\bibliographystyle{amsplain}
%    Insert the bibliography data here.
\baselineskip = 16pt

\section{Introduction}
Let $\mathbb{D}$ denote the open unit disc in the complex plane $\mathbb{C}$ and M\"ob be the group of all the biholomorphic automorphisms of $\mathbb{D}$. Suppose M\"ob $\times$ M\"ob is the direct product of two copies of the group M\"ob. Every $(\varphi_1, \varphi_2) \in \mbox{M\"ob} \times \mbox{M\"ob}$ gives rise to a biholomorphism of $\mathbb{D}^2$ via the map
$$(z_1, z_2) \to (\varphi_1(z_1), \varphi_2(z_2)),\,\,(z_1, z_2) \in \mathbb{D}^2.$$
Note that the biholomorphic automorphism group Aut$(\mathbb{D}^2)$ of $\mathbb{D}^2$ is the semi-direct product of the group M\"ob $\times$ M\"ob and the permutation group $\mathcal{S}_2$ of two elements. The diagonal subgroup $\{(\varphi, \varphi) : \varphi \in \mbox{M\"ob}\}$ of M\"ob $\times$ M\"ob is naturally isomorphic to M\"ob. Therefore, every $\varphi \in \mbox{M\"ob}$ gives rise to a biholomorphic map $\tilde{\varphi}$ of $\mathbb{D}^2$ defined by 
$$\tilde{\varphi}(z_1, z_2) = (\varphi(z_1), \varphi(z_2)),\,\,(z_1, z_2) \in \mathbb{D}^2.$$
Throughout this article, given a $\varphi \in \mbox{M\"ob}$, $\tilde{\varphi}$ represents the biholomorphic map on $\mathbb{D}^2$ defined in the equation above.

Let ${\mathcal M}_n(\mathbb C)$ denote the vector space of all $n\times n$ complex matrices and $\langle ~,~ \rangle_{{\mathbb C}^n}$ be the standard inner product in ${\mathbb C}^n$, where $n$ is an arbitrary but fixed natural number. Suppose $\Omega$ is a bounded domain in $\mathbb{C}^d$ for a natural number $d$. A function $K: \Omega \times\Omega \rightarrow {\mathcal M}_n(\mathbb C)$ is said to be non-negative definite, if 
$$\sum_{i,j=1}^p\langle K(w_i,w_j)\zeta_j,\zeta_i\rangle_{{\mathbb{C}^n}} \geq 0$$
holds for every subsets $\{w_1, \ldots , w_p\}$ of $\Omega$ and $\zeta_1,\ldots,\zeta_p \in{\mathbb C}^n$, $p\in \mathbb{N}$. Due to \cite[Theorem 6.12]{PR}, a non-negative definite function $K: \Omega \times\Omega \rightarrow {\mathcal M}_n(\mathbb C)$ gives rise to a Hilbert space $\mathcal{H}$ consisting of $\mathbb{C}^n$-valued functions on $\Omega$ such that the linear span of $K(\cdot, w)\zeta$ , $w\in \Omega$, $\zeta\in{\mathbb C}^n$, is  dense in $\mathcal{H}$ and $K$ has the reproducing property, namely, $$\langle f,K(\cdot,w)\zeta\rangle =
\langle f(w),\zeta\rangle_{{\mathbb C}^n},\, w\in \Omega.$$
The function $K$ is said to be the reproducing kernel of the Hilbert space $\mathcal{H}$. In general, a reproducing kernel Hilbert space $\mathcal{H}$ with a reproducing kernel $K$ is denoted by $(\mathcal{H}, K)$. Moreover, if $K$ is holomorphic in first $d$-variables and anti-holomorphic in last $d$-variables, then $\mathcal{H}$ consists of $\mathbb{C}^n$-valued holomorphic functions over $\Omega$. Throughout this article, we consider non-negative definite functions  $K: \Omega \times\Omega \rightarrow {\mathcal M}_n(\mathbb C)$ that are holomorphic in the first $d$-variables and anti-holomorphic in the last $d$-variables. 

If $\Omega$ is simply connected, and $K$ is a non-negative definite kernel with $K(z,w) \neq 0$, $z,w\in \Omega$, then, considering the principal branch of the logarithm on $\Omega\times \Omega$, it follows that $K^t$ is holomorphic in the first $d$-variables and anti-holomorphic in the last $d$-variables for any positive real number $t$. However, it is not obvious that it is non-negative definite if $t\not \in \mathbb{N}$. The set 
\[\mathcal{W}_\Omega(K):=\{ t > 0 \mid K^t \text{ is non-negative definite}\}\]  
is known as the Wallach set.

Let $\mathcal{H}$ be a Hilbert space consisting of holomorphic functions on $\Omega$ such that the multiplication operators $M_{z_i}$, $1 \leq i \leq d$, by coordinate functions $z_i$ are bounded on $\mathcal{H}$. The tuple of multiplication operators $\boldsymbol{M} = (M_{z_1}, \ldots, M_{z_d})$ on $\mathcal{H}$ naturally defines a module map $\mathfrak{m}: \mathbb{ C } [ \boldsymbol{ z }] \times \mathcal{H} \to \mathcal{H}$ defined by $\mathfrak{m} (p, f) = p(\boldsymbol{M}) f$, making $\mathcal{H}$ to be a module over the polynomial algebra $\mathbb{C}[\boldsymbol{z}]$. Hilbert modules over a function algebra were introduced by R. G. Douglas, and an extensive study appeared in \cite{RGDM,GDM,RGDVIP,JS}.
A Hilbert space $\mathcal{H}$ is said to be an analytic Hilbert module over the polynomial algebra $\mathbb{C}[\boldsymbol{z}]$ if 
\begin{itemize}
    \item $\mathcal{H}$ consists of holomorphic functions on some bounded domain $\Omega \subset \mathbb{C}^d$,
    \item $\mathbb{C}[\boldsymbol{z}]$ is dense in $\mathcal{H}$,
    \item $\mathcal{H}$ possesses a reproducing kernel on $\Omega$ and
    \item the multiplication operators $M_{z_i}$, $1 \leq i \leq d$, by coordinate functions $z_i$ are bounded on $\mathcal{H}$.
\end{itemize}
A detailed study of such Hilbert modules appeared in \cite{BMS, CG}. A consequence of the polynomial density is that, for every $\boldsymbol{w} = (w_1, \ldots, w_n)$, the dimension of $\bigcap_{i=1}^m \ker (M_{z_i} - w_i)^*$ is $1$ and it is spanned by the element $K(\cdot, \boldsymbol{w})$. Let $\gamma : \Omega^* :=\{\bar{\boldsymbol{w}} : \boldsymbol{w} \in \Omega\} \to \mathcal{H}_K$ be the map defined by $\gamma(\boldsymbol{w}) = K(\cdot, \bar{\boldsymbol{w}})$, $\boldsymbol{w} \in \Omega$. The map $\gamma$ defines a Hermitian holomorphic line bundle $E_K$ on $\Omega^*$. The fiber of $E_K$ at $\bar{\boldsymbol{w}}$ is $\bigcap_{i=1}^m \ker (M_{z_i} - w_i)^*$ and the Hermitian structure on $E_K$ is defined by $K(\boldsymbol{w}, \boldsymbol{w})$. Let 
\begin{equation} \label{eqn:intro1} 
\mathcal{K}(\bar{\boldsymbol{w}}) = -\sum_{ i, j }^{ m }  \frac{ \partial^2 }{ \partial w_i \partial \overline{ w }_j } \log K (\boldsymbol{w},\boldsymbol{w})   dw_i \wedge d\overline{ w}_j,\boldsymbol{w}w \in \Omega,
\end{equation} 
denote the curvature $(1, 1)$ form and 
\begin{equation}\label{eqn:intro2}
 \mathbb{K}(\boldsymbol{w}):= \big(\!\! \big( \frac{ \partial^2 }{ \partial w_i \partial \overline{ w_j } } \log K ( \boldsymbol{w},\boldsymbol{w} ) \big)\! \!\big),  \boldsymbol{w} \in \Omega.
\end{equation}
denote the curvature matrix of the line bundle $E_K$. Cowen and Douglas show that the curvature $(1, 1)$ form $\mathcal{K}$ (equivalently the curvature matrix $\mathbb{K}$) determines the unitary equivalence class of the analytic Hilbert module $\mathcal{H}_K$. 

\begin{definition}\label{def1}
    A commuting pair of bounded linear operators $\boldsymbol{T} = (T_1, T_2)$ on a complex separable Hilbert space $\mathcal{H}$ is said to be \textit{M\"ob-homogeneous} if the Taylor joint spectrum of $\boldsymbol{T}$ lies in $\overline{\mathbb D}^2$ and $\tilde{\varphi}(T_1, T_2) = (\varphi(T_1), \varphi(T_2))$ is unitarily equivalent to $(T_1, T_2)$ for every $\varphi$ in M\"ob. An analytic Hilbert module $\mathcal{H}_K$ consisting of holomoprhic functions over $\mathbb{D}^2$ is said to be M\"ob-homogeneous, if the pair of multiplication operators $\boldsymbol{M} = (M_{z_1}, M_{z_2})$ is M\"ob-homogeneous.
\end{definition}

Let $\lambda$ be a positive real number, $B^{(\lambda)}(z, w) = \frac{1}{(1 - z \bar{w})^{\lambda}},\,\,z, w \in \mathbb{D}$ denote the weighted Bergman kernel of $\mathbb{D}$ and $\mathbb{A}^{(\lambda)}(\mathbb{D})$ denote the weighted Bergman space corresponding to $B^{(\lambda)}$. For any $\lambda, \mu > 0$, $\mathbb{A}^{(\lambda, \mu)}(\mathbb{D}^2) := \mathbb{A}^{(\lambda)}(\mathbb{D}) \otimes \mathbb{A}^{(\mu)}(\mathbb{D})$ is a reproducing kernel Hilbert space with the reproducing kernel $\boldsymbol{B}^{(\lambda, \mu)}(\boldsymbol{z}, \boldsymbol{w}) = \frac{1}{(1 - z_1\bar{w}_1)^{\lambda}(1 - z_2 \bar{w}_2)^{\mu}}$, $\boldsymbol{z} = (z_1, z_2),$ $\boldsymbol{w} = (w_1, w_2) \in \mathbb{D}^2$. Moreover, for each $\lambda, \mu > 0$, $\mathbb{A}^{(\lambda, \mu)}(\mathbb{D}^2)$ is an analytic Hilbert module, known as the weighted Bergman module over $\mathbb{D}^2$. In \cite{DEVH}, it is proved that if $\mathcal{H}_K$ is a  M\"ob $\times$ M\"ob-homogeneous analytic Hilbert module (defined similarly to Definition \ref{def1}), then $\mathcal{H}_K$ is unitarily equivalent to $\mathbb{A}^{(\lambda, \mu)}(\mathbb{D}^2)$ for some $\lambda, \mu > 0$. Also, up to unitary equivalence, $\mathbb{A}^{(\lambda, \lambda)}(\mathbb{D}^2),$ $\lambda > 0$, exhaust all Aut$(\mathbb{D}^2)$-homogeneous analytic Hilbert modules (cf. \cite{DEVH}). Therefore, it is natural to ask if there exists a M\"ob-homogeneous analytic Hilbert module that is not M\"ob $\times$ M\"ob-homogeneous. In this article, we construct three distinct families of mutually unitarily inequivalent families of M\"ob-homogeneous analytic Hilbert modules that are not M\"ob $\times$ M\"ob-homogeneous.

For $\lambda > 0$, let $\mathbb{A}^{(\lambda)}_{\text{\tiny{sim}}}(\mathbb{D}^2)$ and $\mathbb{A}^{(\lambda)}_{\text{\tiny{anti}}}(\mathbb{D}^2)$ denote the set of all symmetric and anti-symmetric functions in $\mathbb{A}^{(\lambda, \lambda)}(\mathbb{D}^2)$, respectively. It follows from \cite[Lemma 4.4]{BGMS} that the spaces $\mathbb{A}^{(\lambda)}_{\text{\tiny{sim}}}(\mathbb{D}^2)$ and $\mathbb{A}^{(\lambda)}_{\text{\tiny{anti}}}(\mathbb{D}^2)$ are reproducing kernel Hilbert spaces over $\mathbb{D}^2$. Let $K^{(\lambda)}_{\text{{\tiny{sim}}}}$ and $K^{(\lambda)}_{\text{{\tiny{anti}}}}$ be the reproducing kernels of $\mathbb{A}^{(\lambda)}_{\text{\tiny{sim}}}(\mathbb{D}^2)$ and $\mathbb{A}^{(\lambda)}_{\text{\tiny{anti}}}(\mathbb{D}^2)$, respectively. For any $\mu, \nu > 0$, consider the non-negative function $K^{(\lambda,\mu, \nu)}: \mathbb{D}^2 \times \mathbb{D}^2 \to \mathbb{C}$, defined by 
\begin{equation*}
K^{(\lambda, \mu, \nu)}(\bl z, \bl w) = \frac{1}{(1-z_1\overbar{w}_1)^{\mu}(1-z_2\overbar{w}_2)^{\nu}} \left[\left(K^{(\lambda)}_{\text{{\tiny{sim}}}}(\bl z, \bl w)\right)^3 + \left(K^{(\lambda)}_{\text{{\tiny{anti}}}}(\bl z, \bl w)\right)^3\right], 
\end{equation*}
$\bl z = (z_1, z_2),\,\,\bl w = (w_1, w_2)\in\mathbb{D}^2$. Let $\mathcal{A}^{(\lambda,\mu,\nu)}$ denote the Hilbert space corresponding to the reproducing kernel $K^{(\lambda,\mu,\nu)}$. In section \ref{sec 3}, it is proved that $\{\mathcal{A}^{(\lambda,\mu,\nu)} : \lambda, \mu, \nu > 0\}$ gives rise to a family of mutually unitarily inequivalent M\"ob-homogeneous analytic Hilbert modules.

For any $\alpha, \beta, \gamma$, let $E_{\alpha, \beta}^{\gamma}$ be a trivial Hermitian holomorphic line bundle over $\mathbb{D}^2$ determined by the holomorphic frame \
\begin{equation}\label{eqn:intro3}
    B_{(\alpha,\beta, \gamma)} (\boldsymbol{z}, \boldsymbol{w}) = (1-z_1\bar{w}_1)^{-\alpha} (1-z_2\bar{w}_2)^{-\beta} (1-z_1\bar{w}_2)^{\gamma} (1-z_2\bar{w}_1)^{\gamma}
\end{equation}
at $\boldsymbol{w} \in \mathbb{D}^2$. In \cite[Theorem 4]{DM}, it is proved that the rank $2$ jet bundle of a Hermitian holomorphic line bundle $E$ over $\mathbb{D}^2$ is homogeneous with respect to the group M\"ob if and only if $E$ is isomorphic to $E_{\alpha, \beta}^{\gamma}$ for some $\alpha, \beta, \gamma$ such that $\alpha, \beta > 0$ and $\alpha \beta > |\gamma|^2$. In section \ref{sec 3}, it is proved that the function $B_{(\alpha,\beta, \gamma)}$ is non-negative definite if $\alpha, \beta > \gamma > 0$. Also, for such $\alpha, \beta, \gamma$, the non-negative definite function $B_{(\alpha,\beta, \gamma)}$ gives rise to a M\"ob-homogeneous analytic Hilbert module $\mathcal{H}^{(\alpha, \beta, \gamma)}$ consisting of holomorphic functions over $\mathbb{D}^2$.

For any $\alpha, \beta > \gamma > 0$, let $\mathcal{B}^{(\alpha, \beta, \gamma)}$ denote the curvature matrix defined via Equation \eqref{eqn:intro2} and for any $\eta > 0$, let $\mathcal{B}^{(\alpha, \beta, \gamma, \eta)} := \big[B_{(\alpha,\beta,\gamma)} (\bl z,\bl w )\big]^{ \eta } \mathcal{B}^{(\alpha, \beta, \gamma)}$. In section \ref{sec 3}, it is proved that $\det \mathcal{B}^{(\alpha, \beta, \gamma, \eta)}$ is non-negative definite and it gives rise to a M\"ob-homogeneous analytic Hilbert module $\mathbb{H}^{(\alpha, \beta, \gamma, \eta)}$ consisting of holomorphic functions over $\mathbb{D}^2$.  

In the final section, it is proved that $\{\mathcal{A}^{(\lambda,\mu,\nu)} : \lambda, \mu, \nu > 0\}$, $\{\mathcal{H}^{(\alpha, \beta, \gamma)} : \alpha, \beta > \gamma > 0\}$ and $\{\mathbb{H}^{(\alpha, \beta, \gamma, \eta)} : \alpha, \beta > \gamma > 0, \eta > 0\}$ are three distinct family of M\"ob-homogeneous analytic Hilbert modules and each of them are unitarily inequivalent with the family of weighted Bergman modules over $\mathbb{D}^2$. It is also proved that none of $B_{(\alpha,\beta, \gamma)}$ is a K\"ahler-Einstein metric.

\section{Curvature criteria}  
Let $(\mathcal{H}, K)$ be a M\"ob-homogeneous analytic Hilbert module. Due to \cite[Theorem 2.1]{BDHKM}, there exists a function $J : \mbox{M\"ob} \times \mathbb{D}^2 \to \mathbb{C} \setminus \{0\}$ such that $J(\varphi, \cdot)$ is a holomorphic function for every $\varphi \in \mbox{M\"ob}$ and 
\begin{equation}\label{eqn:quasi1}
K(\boldsymbol{z}, \boldsymbol{w}) = J(\varphi, \boldsymbol{z}) K(\tilde{\varphi}(\boldsymbol{z}), \tilde{\varphi}(\boldsymbol{w})) \overline{J(\varphi, \boldsymbol{w})}    
\end{equation}
holds for every $\varphi \in \mbox{M\"ob}$, $\boldsymbol{z}, \boldsymbol{w} \in \mathbb{D}^2$. Let $\varphi_z$ be the unique involution of M\"ob mapping $z$ to $0$, that is, $\varphi_z(w) = \frac{z - w}{1 - \bar{z}w},$ $w \in \mathbb{D}$. Taking $\boldsymbol{z} = \boldsymbol{w} = (z_1, z_2)$ and then, replacing $\varphi$ by $\varphi_{z_2}$ in Equation \eqref{eqn:quasi1}, we obtain
\begin{equation}\label{eqn:quasi2}
    K(\boldsymbol{z}, \boldsymbol{z}) = J(\varphi_{z_2}, \boldsymbol{z}) K((\varphi_{z_2}(z_1), 0), (\varphi_{z_2}(z_1), 0)) \overline{J(\varphi_{z_2}, \boldsymbol{z})}.
\end{equation}
For any $k$ in the unit circle $\mathbb{T}$, suppose $\psi_k \in \mbox{M\"ob}$ is defined by $\psi_k(z) = kz,$ $z \in \mathbb{D}$. Replacing $\varphi$ by $\psi_k$ and $\boldsymbol{z}, \boldsymbol{w}$ by $\boldsymbol{0} = (0,0)$ in Equation \eqref{eqn:quasi1}, it follows that $|J(\psi_k, \boldsymbol{0})| = 1$. Furthermore, replacing $\boldsymbol{z}$ and $\boldsymbol{w}$ by $(z_1, 0) \in \mathbb{D}^2$ and $\varphi$ by $\psi_k$ for $k \in \mathbb{T}$ in Equation \eqref{eqn:quasi1}, we obtain 
\begin{equation}\label{eqn:quasi4}
    K((z_1, 0), (z_1, 0)) = J(\psi_k, (z_1, 0)) K((kz_1, 0), (kz_1, 0)) \overline{J(\psi_k, (z_1, 0))}.
\end{equation}

Conversely, assume that $(\mathcal{H}, K)$ is an analytic Hilbert module. Also, assume that there exists a projective cocycle $J : \mbox{M\"ob} \times \mathbb{D}^2 \to \mathbb{C} \setminus \{0\}$ such that $|J(\psi_k, \boldsymbol{0})| = 1$ and Equation \eqref{eqn:quasi4} holds for every $k \in \mathbb{K}$ and $z_1 \in \mathbb{D}$. Then, it follows  from \cite[Theorem 2.5]{BDHKM} that $(\mathcal{H}, K)$ is a M\"ob-homogeneous analytic Hilbert module -- equivalently, Equation \eqref{eqn:quasi1} holds for every $\varphi \in \mbox{M\"ob}$ and $\boldsymbol{z}, \boldsymbol{w} \in \mathbb{D}^2$.

Let $\boldsymbol{M} = (M_{z_1}, M_{z_2})$ be the tuple of multiplication operators by the coordinate functions on $(\mathcal{H}, K)$. Recall that an analytic Hilbert module $(\mathcal{H}, K)$ over $\mathbb{D}^2$ gives rise to a Hermitian holomorphic line bundle $E_K$ over $\mathbb{D}^2$ and the equivalence class of $E_K$ determines the unitary equivalence class of the module $(\mathcal{H}, K)$. A criterion for homogeneity of an analytic Hilbert module consisting of holomorphic functions over a bounded domain $\Omega$, with respect to a subgroup $G$ of the biholomorphic automorphism group of $\Omega$, is provided in \cite[Theorem 2.10]{BDHKM} in terms of the curvature of the associated line bundle $E_K$. In the case of M\"ob-homogeneity for an analytic Hilbert module $(\mathcal{H}, K)$ over $\mathbb{D}^2$, the criterion becomes more transparent, and is provided in the following proposition.

\begin{prop}\label{prop:curv}
    Let $(\mathcal{H}, K)$ be an analytic Hilbert module consisting of holomorphic functions over $\mathbb{D}^2$. Then $(\mathcal{H}, K)$ is M\"ob-homogeneous if and only if for every $z_1, z_2 \in \mathbb{D}$, the curvature matrix $\mathbb{K}$ of the associated line bundle $E_K$ satisfies $\mathbb{K}(\varphi_{z_2}(z_1),0)=\mathbb{K}(k\varphi_{z_2}(z_1),0)$ for every $k \in \mathbb{T}$ and
    \begin{equation}\label{eqn:quasi3}
        \mathbb{K}(z_1,z_2)= D\tilde{\varphi}_{z_2}(z_1,z_2)^t \mathbb{K}(\varphi_{z_2}(z_1),0) \overline{D\tilde{\varphi}_{z_2}(z_1,z_2)},
    \end{equation} 
    where $\varphi_{z_2}$ is the unique involution in M\"ob mapping $z_2$ to $0$, $D\tilde{\varphi}_{z_2}$ denotes the derivative of $\tilde{\varphi}_{z_2}$ and $D\tilde{\varphi}_{z_2}(z_1,z_2)^t$ is the transpose of the derivative matrix $D\tilde{\varphi}_{z_2}(z_1,z_2)$.
\end{prop}

\begin{proof}
Assume that $(\mathcal{H}, K)$ is an analytic Hilbert module. Then, \cite[Proposition 2.7]{BDHKM} implies that
$$\mathbb{K}(z_1, z_2) = D \tilde{\varphi}(z_1, z_2)^t \mathbb{K}(\varphi(z_1), \varphi(z_2)) \overline{D\tilde{\varphi}(z_1, z_2)}$$
holds for every $\varphi \in \mbox{M\"ob}$ and $(z_1, z_2) \in \mathbb{D}^2$. Replacing $\varphi$ by $\varphi_{z_2}$ in the equation above, we obtain Equation \eqref{eqn:quasi3}.

Conversely, assume that $\mathbb{K}(\varphi_{z_2}(z_1),0)=\mathbb{K}(k\varphi_{z_2}(z_1),0)$ and Equation \eqref{eqn:quasi3} holds for every $(z_1, z_2)$ in $\mathbb{D}^2$ and $k \in \mathbb{T}$.
Let $\varphi \in \mob$ be such that $\varphi = \psi_k \circ \varphi_{a}$,
for some $k \in \mathbb{T}$ and $a \in \mathbb{D}$, where $\psi_k \in \mbox{M\"ob}$ is defined by $\psi_k(z) = kz,$ $z \in \mathbb{D}$. A direct computation shows that $\varphi_{\varphi(z_2)} \circ \varphi = \psi_{\tilde{k}} \circ \varphi_{z_2}$, where $\tilde{k} = k \left(\frac{1 - a\bar{z}_2}{1 - \bar{a}z_2}\right) \in \mathbb{T}$. Then,
\begin{align}\label{eqn:curv2}    
\nonumber \mathbb{K}(\varphi z_1,\varphi z_2)&= D\tilde{\varphi}_{\varphi(z_2)}(\varphi(z_1),\varphi(z_2))^{t}  \mathbb{K}(\varphi_{\varphi(z_2)}(\varphi(z_1)),0) \overbar{D\tilde{\varphi}_{\varphi(z_2)}(\varphi(z_1),\varphi(z_2))}\\
\nonumber &= [D(\tilde{\varphi}_{\varphi(z_2)}\circ \tilde{\varphi})(z_1,z_2)  D(\tilde{\varphi}^{-1})(\varphi(z_1), \varphi(z_2))]^t \mathbb{K}(\tilde{k}\varphi_{z_2}(z_1),0)\\
\nonumber &\phantom{xx}\overbar{D(\tilde{\varphi}_{\varphi(z_2)}\circ \tilde{\varphi})(z_1,z_2)}  \overbar{D(\tilde{\varphi}^{-1})(\varphi(z_1), \varphi(z_2))}^{-1}\\
\nonumber &= [D\tilde{\varphi}(z_1,z_2)^{-1}]^t  D\tilde{\varphi}_{z_2}(z_1,z_2)^t D \tilde{\psi}_{\tilde{k}} (\varphi_{z_2}(z_1), 0)^t \mathbb{K}(\tilde{k}\varphi_{z_2}(z_1),0)\\
\nonumber &\phantom{xx} \overline{D \tilde{\psi}_{\tilde{k}} (\varphi_{z_2}(z_1), 0)} \overbar{D\tilde{\varphi}_{z_2}(z_1,z_2)}  \overbar{D\tilde{\varphi}(z_1,z_2)}^{-1}\\
\nonumber &=[D\tilde{\varphi}(z_1,z_2)^{-1}]^t D\tilde{\varphi}_{z_2}(z_1,z_2)^t   \mathbb{K}(\varphi_{z_2}(z_1),0) \overbar{D\tilde{\varphi}_{z_2}(z_1,z_2)}  \overbar{D\tilde{\varphi}(z_1,z_2)}^{-1}\\
&=  [D\tilde{\varphi}(z_1,z_2)^{-1}]^t \mathbb{K}(z_1,z_2) \overbar{D\tilde{\varphi}(z_1,z_2)}^{-1}.
\end{align}  
Here, the fourth equality holds because $D \tilde{\psi}_{\tilde{k}} (\varphi_{z_2}(z_1), 0) = \tilde{k} I_{2 \times 2}$, where $I_{2 \times 2}$ is the $2 \times 2$ identity matrix. 

Since $\tilde{\varphi}(\boldsymbol{M})$ is a bounded operator, it follows that $p \circ \tilde{\varphi} \in \mathcal{H}$ for every $p \in \mathbb{C}[\boldsymbol{z}]$. Therefore, the Hilbert space $\mathcal{H}_{\varphi}$ determined by the non-negative definite function $K_{\varphi} : \mathbb{D}^2 \times \mathbb{D}^2 \to \mathbb{C}$, defined by,
$$K_{\varphi}(z_1, z_2) = K(\varphi^{-1}(z_1), \varphi^{-1}(z_2)),\,\,(z_1, z_2) \in \mathbb{D}^2,$$
contains $\mathbb{C}[\boldsymbol{z}]$ as a dense subset. Also, $\tilde{\varphi}(\boldsymbol{M})$ is unitarily equivalent to the pair of multiplication operators by the coordinate functions $\boldsymbol{M}_{\varphi}$ on $\mathcal{H}_{\varphi}$. Moreover, $\mathcal{H}_{\varphi}$ is an analytic Hilbert module. A direct computation, using the chain rule, implies that the curvature matrix $\mathbb{K}_{\varphi}$ of the associated line bundle of $\mathcal{H}_{\varphi}$ satisfies 
\begin{equation}\label{eqn:curv1}
    \mathbb{K}_{\varphi}(z_1,z_2)  = D\tilde{\varphi}(z_1,z_2) \mathbb{K}(\varphi z_1,\varphi z_2) \overline{D\tilde{\varphi}(z_1,z_2)},
\end{equation}
for every $(z_1, z_2) \in \mathbb{D}^2$. Combining Equations \eqref{eqn:curv2} and \eqref{eqn:curv1}, we obtain $\mathbb{K}_{\varphi} = \mathbb{K}$. Therefore, $\tilde{\varphi}(\boldsymbol{M})$ and $\boldsymbol{M}$ are unitarily equivalent.
\end{proof}

\section{Construction of M\"ob-homogeneous analytic Hilbert modules}\label{sec 3}
In this section, we construct three families of M\"ob-homogeneous analytic Hilbert modules. The section is divided into three subsections, each of which provides the construction of a family of M\"ob-homogeneous analytic Hilbert modules. 

\subsection{The family \texorpdfstring{$\mathcal{A}^{(\lambda, \mu, \nu)}$}{alamdamunu}}
For $\lambda > 0$, recall that $\mathbb{A}^{(\lambda, \lambda)}(\mathbb{D}^2) := \mathbb{A}^{(\lambda)}(\mathbb{D}) \otimes \mathbb{A}^{(\lambda)}(\mathbb{D})$ is the reproducing kernel Hilbert space with the reproducing kernel $\boldsymbol{B}^{(\lambda, \lambda)}(\boldsymbol{z}, \boldsymbol{w}) = \frac{1}{(1 - z_1\bar{w}_1)^{\lambda}(1 - z_2 \bar{w}_2)^{\lambda}}$, $\boldsymbol{z} = (z_1, z_2),$ $\boldsymbol{w} = (w_1, w_2) \in \mathbb{D}^2$. Let $\mathbb{A}^{(\lambda)}_{\text{\tiny{sim}}}(\mathbb{D}^2)$ and $\mathbb{A}^{(\lambda)}_{\text{\tiny{anti}}}(\mathbb{D}^2)$ denote the set of all symmetric and anti-symmetric functions in $\mathbb{A}^{(\lambda, \lambda)}(\mathbb{D}^2)$, respectively. It follows from \cite[Lemma 4.4]{BGMS} that the spaces $\mathbb{A}^{(\lambda)}_{\text{\tiny{sim}}}(\mathbb{D}^2)$ and $\mathbb{A}^{(\lambda)}_{\text{\tiny{anti}}}(\mathbb{D}^2)$ are reproducing kernel Hilbert spaces over $\mathbb{D}^2$. Let $K^{(\lambda)}_{\text{{\tiny{sim}}}}$ and $K^{(\lambda)}_{\text{{\tiny{anti}}}}$ be the reproducing kernels of $\mathbb{A}^{(\lambda)}_{\text{\tiny{sim}}}(\mathbb{D}^2)$ and $\mathbb{A}^{(\lambda)}_{\text{\tiny{anti}}}(\mathbb{D}^2)$, respectively. The reproducing kernels $K^{(\lambda)}_{\text{{\tiny{sim}}}}$ and $K^{(\lambda)}_{\text{{\tiny{anti}}}}$ are given by the formula 
$$K^{(\lambda)}_{\text{{\tiny{sim}}}}(\bl z, \bl w) =
\frac{1}{2}
\left[\frac{1}{(1-z_1\overline{w}_2)^{\lambda}(1-z_2\overline{w}_2)^{\lambda}}+\frac{1}{(1-z_1\overline{w}_2)^{\lambda}(1-z_2\overline{w}_2)^{\lambda}}\right],$$

$$K^{(\lambda)}_{\text{{\tiny{anti}}}}(\bl z, \bl w) = \frac{1}{2}\left[\frac{1}{(1-z_1\overline{w}_1)^{\lambda}(1-z_2\overline{w}_2)^{\lambda}}-\frac{1}{(1-z_1\overline{w}_2)^{\lambda}(1-z_2\overline{w}_1)^{\lambda}}\right]$$
for every $\bl z = (z_1, z_2),\,\,\bl w = (w_1, w_2)\in\mathbb{D}^2$ (cf. \cite[Proposition 4.2]{BGMS},\cite[Proposition 2.2]{MRZ}). Due to \cite[Theorem 5.16]{PR}, product of two non-negative definite kernels is a non-negative definite kernel and therefore, for any $\mu, \nu > 0$, $K^{(\lambda,\mu, \nu)}: \mathbb{D}^2 \times \mathbb{D}^2 \to \mathbb{C}$, defined by 
\begin{equation}\label{eqn:1}
K^{(\lambda, \mu, \nu)}(\bl z, \bl w) = \frac{1}{(1-z_1\overline{w}_1)^{\mu}(1-z_2\overline{w}_2)^{\nu}} \left[\left(K^{(\lambda)}_{\text{{\tiny{sim}}}}(\bl z, \bl w)\right)^3 + \left(K^{(\lambda)}_{\text{{\tiny{anti}}}}(\bl z, \bl w)\right)^3\right], 
\end{equation}
$\bl z = (z_1, z_2),\,\,\bl w = (w_1, w_2)\in\mathbb{D}^2$, is a non-negative definite kernel. Substituting the values of $K^{(\lambda)}_{\text{{\tiny{sim}}}}$ and $K^{(\lambda)}_{\text{{\tiny{anti}}}}$ to the right hand side of Equation \eqref{eqn:1}, we obtain
\begin{align}\label{eqn:2}
 &K^{(\lambda, \mu, \nu)}(\bl z, \bl w) \\
\nonumber &\phantom{xxxx}= \frac{1}{4(1-z_1\overline{w}_1)^{\mu+\lambda}(1-z_2\overline{w}_2)^{\nu+\lambda}}\left[\frac{1}{(1-z_1\overline{w}_1)^{2\lambda}(1-z_2\overline{w}_2)^{2\lambda}}+\frac{3}{(1-z_1\overline{w}_2)^{2\lambda}(1-z_2\overline{w}_1)^{2\lambda}}\right], 
\end{align}
for every $\bl z = (z_1, z_2),\,\,\bl w = (w_1, w_2)\in\mathbb{D}^2$. The non-negative definite function $K^{(\lambda, \mu, \nu)}$ determines a Hilbert space $\mathcal A^{(\lambda, \mu,\nu)}$, consisting of holomorphic functions of $\mathbb{D}^2$, such that $K^{( \lambda, \mu, \nu)}$ is the reproducing kernel of $\mathcal A^{(\lambda, \mu,\nu)}$ (cf. \cite[Theorem 2.14]{PR}).

\begin{prop}\label{prop1}
The multiplication operators by the coordinate functions $M_{z_1}$ and $M_{z_2}$ on $\mathcal{A}^{(\lambda, \mu,\nu)}$ are bounded.    
\end{prop}
\begin{proof}
Since the multiplication operators by the coordinate functions $z_i,$ $i = 1, 2$ are bounded on the weighted Bergman space $ \mathbb{A}^{(\mu,\nu)}(\mathbb{D}^2) := \mathbb{A}^{(\mu)}(\mathbb{D}) \otimes \mathbb{A}^{(\nu)}(\mathbb{D})$, there exist $a_i > 0$ such that 
$$(a_i^2 - z_i \overline{w}_i)\boldsymbol{\boldsymbol{B}}^{(\mu,\nu)}(\bl z, \bl w)$$
is non-negative definite. This implies that 
$$(a_i^2 - z_i \overline{w}_i) K^{(\lambda,\mu,\nu)}(\bl z, \bl w)$$
is non-negative definite. Therefore, for each $i = 1, 2$, the multiplication operators by the coordinate function $M_{z_i}$ is a bounded operator on $\mathcal{A}^{(\lambda, \mu, \nu)}$.
\end{proof}

\begin{prop}\label{prop3}
The set of all polynomials $\mathbb{C}[\bl z]$ is a dense subspace of  $\mathcal{A}^{(\lambda,\mu, \nu)}(\mathbb{D}^2)$.
\end{prop}
\begin{proof}
For $i = 1, 2$, since $M_{z_i}$, the multiplication operator by the coordinate function $z_i$ on $\mathcal{A}^{(\lambda, \mu, \nu)}(\mathbb{D}^2))$, is a bounded operator and $K^{(\lambda, \mu,\nu)}(\cdot, (0,0))$ is the constant function $1$, it follows that $\mathbb{C}[z]$ is contained in $\mathcal{A}^{(\lambda, \mu, \nu)}(\mathbb{D}^2)$.

Note that the reproducing kernel Hilbert space with the reproducing kernel $\left(K^{(\lambda)}_{\text{{\tiny{sim}}}}\right)^3$ (resp. $\left(K^{(\lambda)}_{\text{{\tiny{anti}}}}\right)^3$) contains the set of all symmetric (resp. anti-symmetric) polynomials as a dense subspace (cf. \cite[Proposition 3.4]{GM}). Therefore, the reproducing kernel Hilbert space with the reproducing kernel $\left(K^{(\lambda)}_{\text{{\tiny{sim}}}}\right)^3 +\left(K^{(\lambda)}_{\text{{\tiny{anti}}}}\right)^3$ contains $\mathbb{C}[z]$ as a dense subspace. Again, an application of \cite[Proposition 3.4]{GM} yields that $\mathbb{C}[z]$ is a dense subspace of $\mathcal{A}^{(\lambda,\mu, \nu)}(\mathbb{D}^2)$.
\end{proof}

For every $\lambda, \mu, \nu > 0$, let $c^{(\lambda, \mu,\nu)} : \mbox{M\"ob} \times \mathbb{D}^2 \to \mathbb{C}$ be defined by,
\begin{equation}\label{eqn:3}
c^{(\lambda, \mu,\nu)}(\varphi, (z_1, z_2)) = \left(\varphi'(z_1)\right)^{ \frac{3\lambda +\mu}{2}} \left(\varphi'(z_2)\right)^{\frac{3\lambda +\nu}{2}},\,\,\varphi \in \mbox{M\"ob},\,\,(z_1, z_2) \in \mathbb{D}^2.
\end{equation}
A direct computation shows that the map $c^{(\lambda, \mu, \nu)}$ is a projective cocycle.

\begin{prop}\label{prop2}
Given any $\lambda, \mu, \nu > 0$, the reproducing kernel $K^{(\lambda, \mu,\nu)}$ is quasi-invariant with respect to the projective cocycle $c^{(\lambda, \mu, \nu)}$, that is, 
$$K^{(\lambda, \mu,\nu)}(\bl z, \bl w) = c^{(\lambda, \mu, \nu)}(\varphi, \bl z) K^{(\lambda, \mu,\nu)}(\varphi \bl z, \varphi \bl w) \overline{c^{(\lambda, \mu, \nu)}(\varphi, \bl w)}$$
holds for every $\varphi \in$ M\"ob and $\bl z, \bl w \in \mathbb{D}^2$.
\end{prop}
\begin{proof}
Let $\varphi$ be an arbitrary element of M\"ob and $\bl z = (z_1, z_2)$, $\bl w = (w_1, w_2)$ be arbitrary elements of $\mathbb{D}^2$. Suppose $\varphi(z) = t \frac{z - a}{1 - \bar{a}z},$ $z \in \mathbb{D}$, where $t \in \mathbb{T}$ and $a \in \mathbb{D}$. A direct computation shows that 
\begin{align*}
&(1 - \varphi(z_1) \overline{\varphi(w_1)})^{\mu + \lambda} = \left(\varphi'(z_1)\right)^{\frac{\mu + \lambda}{2}} \left(\overline{\varphi'(w_1)}\right)^{\frac{\mu + \lambda}{2}} (1 - z_1 \bar{w}_1)^{\mu + \lambda} \\
&(1 - \varphi(z_2) \overline{\varphi(w_2)})^{\nu + \lambda} = \left(\varphi'(z_2)\right)^{\frac{\nu + \lambda}{2}} \left(\overline{\varphi'(w_2)}\right)^{\frac{\nu + \lambda}{2}} (1 - z_2 \bar{w}_2)^{\nu + \lambda},\\
&(1 - \varphi(z_i) \overline{\varphi(w_j)})^{2\lambda} = \left(\varphi'(z_i)\right)^{\frac{2\lambda}{2}} \left(\overline{\varphi'(w_j)}\right)^{\frac{2\lambda}{2}} (1 - z_i \bar{w}_j)^{2\lambda},\,\,i, j = 1, 2.  
\end{align*}
 Substituting the above three equations in the expression of $K(\varphi(\bl z), \varphi(\bl w))$, we observe that $K^{(\mu,\nu,\lambda)}(\bl z. \bl w) = c^{(\lambda, \mu,\nu)}(\varphi, \bl z) K^{(\lambda, \mu,\nu)}(\varphi(\bl z), \varphi(\bl w)) \overline{c^{(\lambda, \mu,\nu)}(\varphi, \bl w)}$.
\end{proof}

\begin{cor}\label{mob-homogeneity of A}
    For every $\lambda, \mu, \nu > 0$, the Hilbert space $\mathcal{A}^{(\lambda, \mu, \nu)}$ is a M\"ob-homogeneous analytic Hilbert module.
\end{cor}

\begin{proof}
  A direct consequence of Proposition \ref{prop1} and Proposition \ref{prop3} is that the Hilbert space $\mathcal{A}^{(\lambda, \mu, \nu)}$ is an analytic Hilbert module. M\"ob-homogeneity of $\mathcal{A}^{(\lambda, \mu, \nu)}$ follows from Proposition \ref{prop2}.
\end{proof}

\begin{remark}
Let $n$ be an odd natural number. For any $\lambda, \mu, \nu > 0$, consider the non-negative definite function $K^{(\lambda, \mu, \nu, n)} : \mathbb{D}^2 \times \mathbb{D}^2 \to \mathbb{C}$, defined by
\begin{equation*}
K^{(\lambda, \mu, \nu, n)}(\bl z, \bl w) = \frac{1}{(1-z_1\overline{w}_1)^{\mu}(1-z_2\overline{w}_2)^{\nu}} \left[\left(K^{(\lambda)}_{\text{{\tiny{sim}}}}(\bl z, \bl w)\right)^n + \left(K^{(\lambda)}_{\text{{\tiny{anti}}}}(\bl z, \bl w)\right)^n\right], 
\end{equation*}
for $\bl z = (z_1, z_2),\,\,\bl w = (w_1, w_2)\in\mathbb{D}^2$. Let $\mathcal{A}^{(\lambda, \mu, \nu, n)}$ be the reproducing kernel Hilbert space determined by $K^{(\lambda, \mu, \nu, n)}$. It is easy to verify that Proposition \ref{prop1}, Proposition \ref{prop3} and Proposition \ref{prop2} hold. Therefore, the Hilbert space $\mathcal{A}^{(\lambda, \mu, \nu, n)}$ also determines a M\"ob-homogeneous analytic Hilbert module.
\end{remark}

\subsection{The family \texorpdfstring{$\mathcal{H}^{(\alpha, \beta, \gamma)}$}{alphabetagamma}}

In this subsection, we prove that the holomorphic frame $B_{(\alpha, \beta, \gamma)}$ defined by Equation \eqref{eqn:intro3} of the line bundle $E^{\gamma}_{\alpha, \beta}$, given in \cite[Theorem 4]{DM}, is a non-negative definite function on $\mathbb{D}^2 \times \mathbb{D}^2$ for any real numbers $\alpha, \beta, \gamma$ satisfying $\alpha, \beta > \gamma > 0$. Therefore, for such $\alpha,\beta, \gamma$, the function $B_{(\alpha, \beta, \gamma)}$ determines a reproducing kernel Hilbert space $\mathcal{H}^{(\alpha, \beta, \gamma)}$. We further prove that the Hilbert space $\mathcal{H}^{(\alpha, \beta, \gamma)}$ is a M\"ob-homogeneous analytic Hilbert module.

\begin{prop}\label{propB1}
    For $\gamma>0$ and $\alpha,\beta>\gamma$, the function $B_{(\alpha,\beta,\gamma)}:\mathbb{D}^2 \times \mathbb{D}^2 \to \mathbb{C}$, defined by
\begin{equation}
B_{(\alpha,\beta,\gamma)}\big((z_1,z_2),(w_1,w_2)
\big) =(1-z_1\overline{w}_1)^{-\alpha} (1-z_2\overline{w}_2)^{-\beta} (1-z_1\overline{w}_2)^{\gamma} (1-z_2\overline{w}_1)^{\gamma},
\end{equation}
$\bl z = (z_1, z_2),\,\,\bl w = (w_1, w_2)\in\mathbb{D}^2$, is non-negative definite.
\end{prop}

\begin{proof}
Let $H^2(\mathbb{D}^2)$ be the Hardy space over the bidisc $\mathbb{D}^2$ and $\Delta := \{(z,z) : z \in \mathbb{D}\}$ be the diagonal subset of $\mathbb{D}^2$. Note that $H^2(\mathbb{D}^2)$ is a reproducing kernel Hilbert space with the reproducing kernel $K : \mathbb{D}^2 \times \mathbb{D}^2 \to \mathbb{C}$, defined by
$$K\big((z_1,z_2),(w_1,w_2)\big)=\frac{1}{(1-z_1\overbar{w}_1)(1-z_2\overbar{w}_2)},\,\, (z_1, z_2),(w_1, w_2) \in \mathbb{D}^2.$$
Suppose $\mathcal{A} := \{f \in H^2(\mathbb{D}^2) : f|_{\Delta} = 0\}$ and $\mathcal{S}$ is the orthogonal complement of $\mathcal{A}$ in $H^2(\mathbb{D}^2)$. 

For $k \geq 0$, let $e_k : \mathbb{D}^2 \to \mathbb{C}$ be the function, defined by 
$$e_k(z_1, z_2) = \frac{1}{\sqrt{k + 1}} \frac{z_1^{k+1} - z_2^{k+1}}{z_1 - z_2},\,\,(z_1, z_2) \in \mathbb{D}^2.$$
It is easy to verify that $\{e_k : k \geq 0\}$ is a complete orthonormal set of $\mathcal{S}$. If $K_{\mathcal{S}}$ is the reproducing kernel of $\mathcal{S}$, then
\begin{align*}
K_{\mathcal{S}}\big((z_1, z_2), (w_1, w_2)\big)
&=\displaystyle \sum_{k=0}^{\infty}\frac{1}{k+1}e_k(z_1,z_2) \overline{e_k(w_1, w_2)} \\
&=\displaystyle \sum_{k=0}^{\infty}\frac{1}{k+1}\left(\frac{z_1^{k+1}-z_2^{k+1}}{z_1-z_2}\right) \left(\frac{\overline{w}_1^{k+1} - \overline{w}_2^{k+1}}{\overline{w}_1 - \overline{w}_2}\right)\\
% &=\frac{1}{|z_1-z_2|^2} \displaystyle \sum_{p=1}^{\infty} \frac{|z_1^p-z_2^p|^2}{p}\\
&=\frac{1}{(z_1 - z_2)(\overline{w}_1 - \overline{w}_2)}
\log\!\left[
\frac{(1 - z_1 \overline{w}_2)(1 - z_2 \overline{w}_1)}
{(1 - z_1\overline{w}_1)(1 - z_2\overline{w}_2)}
\right]
\end{align*}
Let $\Lambda = \{((z_1,z_2),(w_1,w_2))\in \mathbb{D}^2\times\mathbb{D}^2: z_1\ne z_2, w_1\ne w_2\}$ and $L : \Lambda \times \Lambda \to \mathbb{C}$ be the function defined by
$$L\big((z_1, z_2), (w_1, w_2)\big) = \log\!\left[
\frac{(1 - z_1 \overline{w}_2)(1 - z_2 \overline{w}_1)}
{(1 - z_1 \overline{w}_1)(1 - z_2 \overline{w}_2)} 
\right],\,\,(z_1, z_2), (w_1, w_2) \in \Lambda.$$
The non-negative definiteness of $K_{\mathcal{S}}$ on $\mathbb{D}^2 \times \mathbb{D}^2$ implies that the function $L$ is also a non-negative definite function on $\Lambda \times \Lambda$. Hence, for any $\gamma > 0$, the function $\tilde{L}^{(\gamma)} : \Lambda \times \Lambda \to \mathbb{C}$, defined by 
$$\tilde{L}^{(\gamma)}\big((z_1, z_2), (w_1, w_2)\big) = \left[\frac{(1 - z_1 \overline{w}_2)(1 - z_2 \overline{w}_1)}
{(1 - z_1 \overline{w}_1)(1 - z_2 \overline{w}_2)}\right]^{\gamma},\,\,(z_1, z_2), (w_1, w_2) \in \Lambda,$$ 
is non-negative definite. Since $\bar{\Lambda} = \mathbb{D}^2\times\mathbb{D}^2$, it follows from \cite[Page 362]{AR} that $\tilde{L}^{(\gamma)}$ is a non-negative definite function on $\mathbb{D}^2 \times \mathbb{D}^2$. Hence, for any $\alpha, \beta > \gamma$, $B_{(\alpha, \beta, \gamma)}$, being the Schur product of two non-negative definite functions, is a non-negative definite function on $\mathbb{D}^2 \times \mathbb{D}^2$ (cf \cite[Theorem 4.8]{PR}).
\end{proof}

Due to the Moore-Aronszajn Theorem (\cite[Theorem 2.14]{PR}), the non-negative definite function $B_{(\alpha,\beta,\gamma)}$ defines a Hilbert space $ \mathcal{H}^{(\alpha,\beta,\gamma)}$ consisting of holomorphic functions of $\mathbb{D}^2$ such that $B_{(\alpha,\beta,\gamma)}$ is the reproducing kernel of $\mathcal H^{(\alpha,\beta,\gamma)}$.

\begin{remark}\label{remB}
Let, $\alpha,\beta,\gamma$ be same as it is given in Proposition $\ref{propB1}$. For any positive number $\eta$, we have $\gamma\eta>0$ and $\alpha\eta, \beta\eta>\gamma\eta$. Furthermore, we have $[B_{(\alpha,\beta,\gamma)}]^{\eta} = B_{(\alpha  \eta, \beta \eta, \gamma \eta)}$. Hence, $[B_{(\alpha,\beta,\gamma)}]^{\eta}$ is a non-negative definite function. Thus, any positive power of $B_{(\alpha,\beta,\gamma)}$ is also a non-negative definite function and therefore, $B_{(\alpha, \beta, \gamma)}$ is an infinitely divisible kernel. Also, note that $[B_{(\alpha,\beta,\gamma)}]^{\eta}$ is not a non-negative definite function if $\eta < 0$.  Hence, the Wallach set $\mathcal{W}(B_{(\alpha, \beta, \gamma)})$ of $B_{(\alpha, \beta, \gamma)}$ is  $[0 ,\infty)$.
\end{remark}
% If $K^{-1}(\bl z, \bl w)= (1-z_1 \bar w_1)^\alpha (1-z_2 \bar w_2)^\beta (1-z_1 \bar w_2)^{-\gamma} (1-z_2 \bar w_1)^{-\gamma}$ is a non negative definite function. Then $\frac{1}{(1-z_1 \bar w_1)^\alpha}  \frac{1}{(1-z_2 \bar w_2)^\beta} K^{-1}(\bl z, \bl w)= \frac{1}{(1-z_1 \bar w_2)^\gamma (1-z_2 \bar w_1)^\gamma}$ is a non negative definite function. Let, $\hat{K}(\bl z, \bl w)=\frac{1}{(1-z_1 \bar w_2)^\gamma (1-z_2 \bar w_1)^\gamma}=\displaystyle \sum_{k,l=0}^{\infty} \frac{(\gamma)_k (\gamma)_l}{k! l!} z_1^k z_2^l \bar w_1^l \bar w_2^k$.\\
% So, $\big<\bar\partial_1^n \bar\partial_2^m \hat K_{(0,0)},\bar\partial_1^n \bar\partial_2^m \hat K_{(0,0)}\big>= \partial_1^n \partial_2^m\bar\partial_1^n \bar\partial_2^m \hat K((0,0),(0,0))$\\
% Now, $\partial_1^n \partial_2^m\bar\partial_1^n \bar\partial_2^m \hat K ((z_1,z_2),(w_1,w_2))=\sum_{k,l=0}^{\infty} \frac{(\gamma)_k (\gamma)_l}{k! l!} \partial_1^n \partial_2^m\bar\partial_1^n \bar\partial_2^m(z_1^k z_2^l \bar w_1^l \bar w_2^k) $.\\
% Therefore, $||\bar\partial_1^n \bar\partial_2^m \hat K_{(0,0)}||^2=0$ that is a contradiction.

\begin{prop}\label{propB2}
The multiplication operators by the coordinate functions $M_{z_1}$ and $M_{z_2}$ on $\mathcal H^{(\alpha,\beta,\gamma)}$ are bounded.    
\end{prop}
\begin{proof}
    Since the multiplication operators by the coordinate functions $z_i,$ $i = 1, 2$ are bounded on the weighted Bergman space $\mathbb{A}^{(\alpha-\gamma)}(\mathbb{D})\otimes\mathbb{A}^{(\beta-\gamma)}(\mathbb{D})$, there exist $a_i > 0$ such that 
$$(a_i^2 - z_i \overline{w}_i)\boldsymbol{B}^{(\alpha - \gamma, \beta - \gamma)}\big((z_1, z_2), (w_1, w_2)\big)$$
is a non-negative definite function on $\mathbb{D}^2 \times \mathbb{D}^2$. Hence, it follows that 
$$(a_i^2 - z_i \overline{w}_i) B_{(\alpha,\beta,\gamma)}\big((z_1, z_2), (w_1, w_2)\big)$$
is a non-negative definite function on $\mathbb{D}^2 \times \mathbb{D}^2$. Therefore, for each $i = 1, 2$, the multiplication operators by the coordinate function $M_{z_i}$ is a bounded operator on $\mathcal{H}^{(\alpha,\beta,\gamma)}$.
\end{proof}

\begin{prop}\label{propB3}
The set of all polynomials $\mathbb{C}[\bl z]$ is a dense subspace of  $\mathcal H^{(\alpha,\beta,\gamma)}$.
\end{prop}
\begin{proof}
For $i = 1, 2$, since $M_{z_i}$, the multiplication operator by the coordinate function $z_i$ on $\mathcal H^{(\alpha,\beta,\gamma)}$, is a bounded operator and $B_{(\alpha,\beta,\gamma)}(\cdot, (0,0))$ is the constant function $1$, it follows that $\mathbb{C}[z]$ is contained in $\mathcal H^{(\alpha,\beta,\gamma)}$.
% Now, $(1 - z_1 \bar{w}_1)^{-\alpha}
% = \displaystyle\sum_{m=0}^{\infty} \binom{\alpha+m-1}{m} (z_1 \bar{w}_1)^m,\,\,
% (1 - z_2 \bar{w}_2)^{-\beta}
% = \displaystyle\sum_{n=0}^{\infty}  \binom{\alpha+n-1}{n} (z_2 \bar{w}_2)^n,$ \\ and
% $(1 - z_1 \bar{w}_2)^{\gamma}
% = \displaystyle\sum_{k=0}^{\infty} \binom{\gamma}{k} (-1)^k (z_1 \bar{w}_2)^k,\,\, 
% (1 - z_2 \bar{w}_1)^{\gamma}
% = \displaystyle\sum_{p=0}^{\infty} \binom{\gamma}{p} (-1)^p (z_2 \bar{w}_1)^p$\\
Also, we have
\begin{align*}
&B_{(\alpha,\beta,\gamma)}\big((z_1,z_2),(w_1,w_2)\big) \\
&=\sum_{p,n,k,p\ge 0}
\binom{\alpha+p-1}{p}
\binom{\beta+q-1}{q} 
\binom{\gamma}{r}
\binom{\gamma}{s}
(-1)^{r+s}
z_1^{p+r} z_2^{q+s}
\bar w_1^{p+s}\bar w_2^{q+r}
\end{align*}
for every $(z_1, z_2)$ and $(w_1, w_2)$ in $\mathbb{D}^2$. Note that the set \{$\bar{\partial}_{1}^{m}\bar{\partial}_{2}^{n} B_{(\alpha,\beta,\gamma)}(\cdot, (0,0)) : m,n \geq 0$\} is dense in $\mathcal{H}^{(\alpha,\beta, \gamma)}$.
% Since, $\bar{\partial_1}^m \bar{w_1}^k=
% \begin{cases}
%     \dfrac{k!}{(k-m)!}\bar{w_1}^{k-m} & \text{if } k \geq m\\
%     0 & \text{if } k < m
% \end{cases}$
% and similarly for $\bar{\partial_2}^n$. After applying derivatives and evaluating at $(w_1,w_2)=(0,0)$ only terms survive where $p+s=m$ and $q+r=n$. Now, setting $p=m-s$ and $q=n-r$ with $0\le s \le m,\,0\le r\le n$\\
A direct computation shows that
\begin{align*}
\bar{\partial}_{1}^{m}\bar{\partial}_{2}^{n} B_{(\alpha,\beta,\gamma)}&\big(\cdot, (0,0)\big)\\ &= m!n!\sum_{r=0}^{n}\sum_{s=0}^{m} \binom{\alpha+m-s-1}{m-s} \binom{\beta+n-r-1}{n-r} \binom{\gamma}{r} \binom{\gamma}{s} (-1)^{r+s} z_1^{m+r-s} z_2^{n+s-r}
\end{align*}
Hence polynomials are dense in $\mathcal{H}^{(\alpha,\beta, \gamma)}$. 
\end{proof}

For every $\gamma > 0$ and $\alpha,\beta>\gamma$, let $J^{(\alpha,\beta,\gamma)} : \mbox{M\"ob} \times \mathbb{D}^2 \to \mathbb{C}$ be defined by,
\begin{equation*}
J^{(\alpha,\beta,\gamma)}(\varphi, (z_1, z_2)) = \left(\varphi'(z_1)\right)^{ \frac{\gamma-\alpha}{2}} \left(\varphi'(z_2)\right)^{\frac{\gamma-\beta}{2}},\,\,\varphi \in \mbox{M\"ob},\,\,(z_1, z_2) \in \mathbb{D}^2.
\end{equation*}
It is easy to verify that the map $J^{(\alpha,\beta,\gamma)}$ is a projective cocycle.

\begin{prop}\label{propB4}
For $\gamma>0$ and $\alpha,\beta>\gamma$, the reproducing kernel $B_{(\alpha,\beta,\gamma)}$ is quasi-invariant with respect to the projective cocycle $J^{(\alpha,\beta,\gamma)}$, that is, 
$$B_{(\alpha,\beta,\gamma)}(\bl z, \bl w) = J^{(\alpha,\beta,\gamma)}(\varphi, \bl z) B_{(\alpha,\beta,\gamma)}(\varphi \bl z, \varphi \bl w) \overline{J^{(\alpha,\beta,\gamma)}(\varphi, \bl w)}$$
holds for every $\varphi \in$ M\"ob and $\bl z, \bl w \in \mathbb{D}^2$.
\end{prop}
\begin{proof}
    Let $\varphi$ be an arbitrary element of M\"ob and $\bl z = (z_1, z_2)$, $\bl w = (w_1, w_2)$ be arbitrary elements of $\mathbb{D}^2$. Suppose $\varphi(z) = t \frac{z - a}{1 - \bar{a}z},$ $z \in \mathbb{D}$, where $t \in \mathbb{T}$ and $a \in \mathbb{D}$. It is easy to verify that 
\begin{align*}
&(1 - \varphi(z_1) \overline{\varphi(w_1)})^{-\alpha} = \left(\varphi'(z_1)\right)^{\frac{-\alpha}{2}} \left(\overline{\varphi'(w_1)}\right)^{\frac{-\alpha}{2}} (1 - z_1 \bar{w}_1)^{-\alpha} \\
&(1 - \varphi(z_2) \overline{\varphi(w_2)})^{-\beta} = \left(\varphi'(z_2)\right)^{\frac{-\beta}{2}} \left(\overline{\varphi'(w_2)}\right)^{\frac{-\beta}{2}} (1 - z_2 \bar{w}_2)^{-\beta}\,\,\mbox{and}\\
&(1 - \varphi(z_i) \overline{\varphi(w_j)})^{\gamma} = \left(\varphi'(z_i)\right)^{\frac{\gamma}{2}} \left(\overline{\varphi'(w_j)}\right)^{\frac{\gamma}{2}} (1 - z_i \bar{w}_j)^{\gamma},\,\,i, j = 1, 2.
\end{align*}
 Substituting above three equations in the expression of $B_{(\alpha,\beta,\gamma)}(\varphi(\bl z), \varphi(\bl w))$, we observe that $B_{(\alpha,\beta,\gamma)}(\bl z, \bl w) = J^{(\alpha,\beta,\gamma)}(\varphi, \bl z) B_{(\alpha,\beta,\gamma)}(\varphi \bl z, \varphi \bl w) \overline{J^{(\alpha,\beta,\gamma)}(\varphi, \bl w)}$.
\end{proof}

\begin{cor}\label{mob homogeneity of H}
    For any real numbers $\alpha, \beta, \gamma$ satisfying $\alpha, \beta > \gamma > 0$, the Hilbert space $\mathcal{H}^{(\alpha, \beta, \gamma)}$ is a M\"ob-homogeneous analytic Hilbert module.
\end{cor}

\begin{proof}
  Let $\alpha, \beta, \gamma$ be real numbers such that $\alpha, \beta > \gamma > 0$. Due to Proposition \ref{propB2} and Proposition \ref{propB3}, it follows that the Hilbert space $\mathcal H^{(\alpha,\beta,\gamma)}$ is an analytic Hilbert module. Finally, M\"ob-homogeneity of $\mathcal{H}^{(\alpha, \beta, \gamma)}$ is a direct consequence of Proposition \ref{propB4}. 
\end{proof}

\begin{remark}
  For any $\eta > 0$, the Hilbert space $\mathcal{H}^{(\alpha \eta, \beta \eta, \gamma \eta)}$, determined by the non-negative definite function $(B_{(\alpha, \beta, \gamma)})^{\eta} = B_{(\alpha \eta, \beta \eta, \gamma \eta)}$, is an analytic Hilbert module. Hence, the analytic Wallach set (cf. \cite[Page 5]{BDHKM}) 
  $$\mathcal{W}_{a}(B_{(\alpha, \beta, \gamma)}) := \{\eta : (B_{(\alpha, \beta, \gamma)})^{\eta}\,\,\mbox{gives rise to an analytic Hilbert module}\}$$
  of the reproducing kernel $B_{(\alpha, \beta, \gamma)}$  is $(0, \infty)$.  
\end{remark}

\subsection{The family \texorpdfstring{$\mathbb{H}^{(\alpha, \beta, \gamma, \eta)}$}{alphabetagammaeta}}

Let $\mathcal{B}^{(\alpha, \beta, \gamma)}$ denote the curvature matrix of $B_{(\alpha, \beta, \gamma)}$. Due to \cite[Corollary 2.5]{GM}, it follows that $\mathcal{B}^{(\alpha, \beta, \gamma)}$ is a non-negative definite function taking values in the set of all $2 \times 2$ complex matrices $\mathcal{M}_2(\mathbb{C})$. Therefore, for any $\eta > 0$, the function $\mathcal{B}^{(\alpha, \beta, \gamma, \eta)} : \mathbb{D}^2 \times \mathbb{D}^2 \to \mathcal{M}_2(\mathbb{C})$, defined by 

\begin{equation}\label{eqn5.1}
\mathcal{ B }^{ (\alpha,\beta,\gamma,\eta) } (\bl z,\bl w ) := \big[B_{(\alpha,\beta,\gamma)} (\bl z,\bl w )\big]^{ \eta } \left( \!\! \left( \frac{ \partial^2 }{ \partial z_i \partial \overline{ w_j } } \log B_{(\alpha,\beta,\gamma)} (\bl z,\bl w ) \right) \!\! \right)_{ i, j = 1 }^2, ~~~ \bl z,\bl w \in \mathbb{D}^2,
\end{equation} 
is a non-negative definite function.
From \cite[Lemma 2.13]{BDHKM}, it follows that $\det \mathcal{B}^{(\alpha, \beta, \gamma, \eta)}$ is a non-negative definite function and therefore, by Moore-Aronszajn Theorem $\det \mathcal{B}^{(\alpha, \beta, \gamma, \eta)}$ gives rise to a Hilbert space $\mathbb{H}^{(\alpha, \beta, \gamma, \eta)}$ consisting of holomorphic functions over $\mathbb{D}^2$. Also, $\mathbb{H}^{(\alpha, \beta, \gamma, \eta)}$ is a $\mob$-homogeneous analytic Hilbert module, thanks to \cite[Theorem 2.16]{BDHKM}. 

\begin{prop}\label{prop 4.6}
Let $\mathcal{B}^{(\alpha,\beta,\gamma)}_{i\bar{j}}$ $1 \leq i, j\leq 2$, denote the $(i, j)$th entry of the curvature matrix $\mathcal{B}^{(\alpha,\beta,\gamma)}$ of the analytic Hilbert module $\mathcal{H}^{(\alpha,\beta,\gamma)}$. Then, for any $\bl z=(z_1,z_2) \in \mathbb{D}^2$, we have
$$\mathcal{B}^{(\alpha,\beta,\gamma)}_{1\bar{1}}(\bl z)=\dfrac{\alpha}{(1-|z_1|^2)^2} \,\,\text{and}\,\,  \mathcal{B}^{(\alpha,\beta,\gamma)}_{1\bar{2}}(\bl z)= -\dfrac{\gamma}{(1-z_1\overline{z}_2)^2}$$
$$\mathcal{B}^{(\alpha,\beta,\gamma)}_{2\bar{1}}(\bl z)= -\dfrac{\gamma}{(1-z_2\overline{z}_1)^2}\,\,\text{and}\,\,  \mathcal{B}^{(\alpha,\beta,\gamma)}_{2\bar{2}}(\bl z)=\dfrac{\beta}{(1-|z_2|^2)^2}$$
\end{prop}

\begin{proof}
For $1 \leq i, j \leq 2$, $\bl z = (z_1, z_2) \in \mathbb{D}^2$, we have
\begin{align*}
\mathcal{B}_{i \bar{j}}^{(\alpha,\beta,\gamma)}(\bl z)
%&= \frac{\partial^2}{\partial z_i \partial \bar{z}_j} \log B^{(\alpha,\beta,\gamma)}(\boldsymbol{z}, \boldsymbol{z})\\
&=-\alpha\frac{\partial^2}{\partial z_1 \partial \overline{z}_1} \log(1-z_1\overline{z}_1)-\beta \frac{\partial^2}{\partial z_2 \partial \overline{z}_2} \log(1-z_2\overline{z}_2)\\ 
& \phantom{xx}+\gamma \frac{\partial^2}{\partial z_1 \partial \overline{z}_2} \log(1-z_1\overline{z}_2)+\gamma \frac{\partial^2}{\partial z_2 \partial \overline{z}_1} \log(1-z_2\overline{z}_1)
\end{align*}
Note that $\frac{\partial^2}{\partial z_i \partial \overline{z}_j}\log(1-z_i\overline{z}_j)=\frac{-1}{(1-z_i\overline{z}_j)^2}$. Then, substituting the values of $\frac{\partial^2}{\partial z_i \partial \overline{z}_j}\log(1-z_i\overline{z}_j)$ in the expression of $\mathcal{B}_{i \bar{j}}^{(\alpha,\beta,\gamma)}(\bl z)$, we obtain the expression of $\mathcal{B}_{i \bar{j}}^{(\alpha,\beta,\gamma)}(z_1, z_2)$ as stated in the statement of this proposition.
\end{proof}

It follows from Proposition \ref{prop 4.6} that
\begin{align}\label{eqn:dettran1}
\nonumber \det \mathcal{ B }^{ (\alpha,\beta,\gamma,\eta) } (\bl z, \bl w) = &\,  (1-z_1\overline{w}_1)^{-2\alpha\eta} (1-z_2\overline{w}_2)^{-2\beta\eta} (1-z_1\overline{w}_2)^{2\gamma\eta} (1-z_2\overline{w}_1)^{2\gamma\eta}  \\ &\times \left[\frac{\alpha\beta}{(1-z_1\overline{w}_1)^2(1-z_2\overline{w}_2)^2}-\frac{\gamma^2}{(1-z_1\overline {w}_2)^2 (1-z_2\overline w_1)^2}\right], 
\end{align}
for $\boldsymbol{z} = (z_1, z_2)$ and $\boldsymbol{w} = (w_1, w_2)$ in $\mathbb{D}^2$.

\section{Inequivalence}
In this section, we first prove that the analytic Hilbert modules in each of the families $\{\mathcal{A}^{(\lambda, \mu, \nu)} : \lambda, \mu, \nu > 0\}$ and $\{\mathcal{H}^{(\alpha, \beta, \gamma)} : \alpha, \beta > \gamma > 0\}$ are mutually unitarily inequivalent. We also prove that the unitary equivalence class of each of the Hilbert modules in the family $\{\mathbb{H}^{(\alpha, \beta, \gamma, \eta)} : \alpha, \beta > \gamma > 0, \eta > 0\}$ depends on the ratio $\frac{\alpha}{\beta}$, $\frac{\beta}{\gamma}$ and $\gamma \eta$. Moreover, we prove that the families are pairwise distinct up to unitary equivalence, and no module in any of these families is unitarily equivalent to a weighted Bergman module over the bidisc.

% Therefore, $\mathcal{A}^{(\lambda, \mu, \nu)}$ gives rise to an Hermitian holomorphic line bundle $E_{(\lambda, \mu, \nu)}$ such that the Hermitian structure of $E_{(\lambda, \mu, \nu)}$ is given by $K^{(\lambda, \mu, \nu)}(\bl w, \bl w)$ and the unitary equivalence class of the analytic Hilbert module $\mathcal{A}^{(\lambda, \mu, \nu)}$ is determined by the curvature of $E_{(\lambda, \mu, \nu)}$. 

\begin{prop}\label{inequiv:prop4}
Let $\mathcal{K}^{(\lambda, \mu, \nu)}_{i\bar{j}}$ $1 \leq i, j\leq 2$, denote the $(i, j)$th entry of the curvature matrix $\mathbb{K}^{(\lambda, \mu, \nu)}$ of the analytic Hilbert module $\mathcal{A}^{(\lambda, \mu,\nu)}$. Then, for any $z \in \mathbb{D}$, we have
\begin{align*}
   &\mathcal{K}_{1\bar{1}}^{(\lambda,\mu,\nu)}(z,0)=\frac{\mu+\lambda}{(1-|z|^2)^2}+\frac{2\lambda}{(1-|z|^2)^2(1+3(1-|z|^2)^{2\lambda})}\left[1+6\lambda |z|^2\frac{(1-|z|^2)^{2\lambda}}{(1+3(1-|z|^2)^{2\lambda})}\right],\\
   &\mathcal{K}_{2\bar{1}}^{(\lambda,\mu,\nu)}(z,0) = \mathcal{K}_{1\bar{2}}^{(\lambda,\mu,\nu)}(z,0)=\frac{6\lambda(1-|z|^2)^{2\lambda}}{(1+3(1-|z|^2)^{2\lambda})} \left[1-\frac{2\lambda|z|^2}{(1-|z|^2)(1+3(1-|z|^2)^{2\lambda})} \right]\,\,\mbox{and}\\
   &\mathcal{K}_{2\bar{2}}^{(\lambda,\mu,\nu)}(z,0)=\nu+\lambda+\frac{2\lambda}{(1+3(1-|z|^2)^{2\lambda})}\left[1+6\lambda |z|^2\frac{(1-|z|^2)^{2\lambda}}{(1+3(1-|z|^2)^{2\lambda})}\right]
\end{align*}
\end{prop}

\begin{proof}
We have, 
\begin{align*}
&K^{(\lambda, \mu, \nu)}(\bl z, \bl w) \\
\nonumber &\phantom{xxxx}= \frac{1}{4(1-z_1\overbar{w}_1)^{\mu+\lambda}(1-z_2\overbar{w}_2)^{\nu+\lambda}}\left[\frac{1}{(1-z_1\overbar{w}_1)^{2\lambda}(1-z_2\overbar{w}_2)^{2\lambda}}+\frac{3}{(1-z_1\overbar{w}_2)^{2\lambda}(1-z_2\overbar{w}_1)^{2\lambda}}\right]\\
\nonumber &\phantom{xxxx}= \frac{1}{4(1-z_1\overbar{w}_1)^{\mu+\lambda}(1-z_2\overbar{w}_2)^{\nu+\lambda}}\times \frac{3}{(1-z_1\overbar{w}_2)^{2\lambda}(1-z_2\overbar{w}_1)^{2\lambda}}\\
& \phantom{xxxxxx} \times\left[1+\frac{(1-z_1\overbar{w}_2)^{2\lambda}(1-z_2\overbar{w}_1)^{2\lambda}}{3(1-z_1\overbar{w}_1)^{2\lambda}(1-z_2\overbar{w}_2)^{2\lambda}}\right], 
\end{align*}
For $1 \leq i, j \leq 2, \bl z \in \mathbb{D}^2$, note that

\begin{align*}
 \mathcal{K}_{i \bar{j}}^{(\lambda, \mu, \nu)}(\bl z) &= \frac{\partial^2}{\partial z_i \partial \bar{z}_j} \log K^{(\lambda, \mu, \nu)}(\boldsymbol{z}, \boldsymbol{z}) \\
 &=  \frac{\partial^2}{\partial z_i \partial \bar{z}_j} \log G(\boldsymbol{z}, \boldsymbol{z})+\frac{\partial^2}{\partial z_i \partial \bar{z}_j} \log\left[1+\frac{(1-z_1\overbar{z}_2)^{2\lambda}(1-z_2\overbar{z}_1)^{2\lambda}}{3(1-z_1\overbar{z}_1)^{2\lambda}(1-z_2\overbar{z}_2)^{2\lambda}}\right].
\end{align*}
where $G(\bl z, \bl z)=\dfrac{3}{4(1-z_1\overbar{z}_1)^{\mu+\lambda}(1-z_2\overbar{z}_2)^{\nu+\lambda}(1-z_1\overbar{z}_2)^{2\lambda}(1-z_2\overbar{z}_1)^{2\lambda}} $.\\
Replacing $\bl z$ by $(z, 0),$ $z \in \mathbb{D},$  we observe that that 
\begin{align*}
   & \frac{\partial^2}{\partial z_1 \partial \bar{z}_1}\log G((z,0),(z,0)) = \phantom{x}\frac{\mu+\lambda}{(1-|z|^2)^2},\\
    &\frac{\partial^2}{\partial z_1 \partial \bar{z}_2}\log G((z,0),(z,0))= 2 \lambda,\\
    &\frac{\partial^2}{\partial z_2 \partial \bar{z}_1}\log G((z,0),(z,0))= 2\lambda,\\
    &\frac{\partial^2}{\partial z_2 \partial \bar{z}_2}\log G(z,0),(z,0))=\nu+\lambda
\end{align*}
Consider, $F(z_1,z_2)=\left[1+\dfrac{(1-z_1\overbar{z}_2)^{2\lambda}(1-z_2\overbar{z}_1)^{2\lambda}}{3(1-z_1\overbar{z}_1)^{2\lambda}(1-z_2\overbar{z}_2)^{2\lambda}}\right]$. Therefore, 
\begin{align*}
    \frac{\partial^2}{\partial z_i \partial \bar{z}_j} \log F(z_1,z_2)= \dfrac{F(z_1,z_2) \frac{\partial^2}{\partial z_i \partial \bar{z}_j}F(z_1,z_2)-\frac{\partial}{\partial z_i} F(z_1,z_2)\frac{\partial}{\partial \bar z_j}F(z_1,z_2)}{F(z_1,z_2)^2}
\end{align*}

For an arbitrary $z \in \mathbb{D}$, a direct computation shows that
\begin{align*}
& F(z,0) = \frac{1+3(1-|z|^2)^{2\lambda}}{3(1-|z|^2)^{2\lambda}}\\
&\frac{\partial}{\partial z_1} F(z,0) = \dfrac{2\lambda \bar z}{3(1-|z|^2)^{2\lambda+1}}\\
&\frac{\partial}{\partial z_2} F(z,0) = \dfrac{-2\lambda \bar z}{3(1-|z|^2)^{2\lambda}}\\
&\frac{\partial}{\partial \bar{z}_1}F(z,0) =\dfrac{2\lambda  z}{3(1-|z|^2)^{2\lambda+1}}\\
& \frac{\partial}{\partial \bar{z}_2} F(z,0) = \dfrac{-2\lambda  z}{3(1-|z|^2)^{2\lambda}}\\
&\frac{\partial^2}{\partial z_1 \partial \bar{z}_1} F(z,0) = \dfrac{2\lambda(1+2\lambda|z|^2)}{3(1-|z|^2)^{2\lambda+2}}\\
&\frac{\partial^2}{\partial z_1 \partial \bar{z}_2} F(z,0) =\dfrac{2\lambda(|z|^2-2\lambda|z|^2-1)}{3(1-|z|^2)^{2\lambda+1}}\\
&\frac{\partial^2}{\partial z_2 \partial \bar{z}_1} F(z,0) = \dfrac{2\lambda(|z|^2-2\lambda|z|^2-1)}{3(1-|z|^2)^{2\lambda+1}} \\
&\frac{\partial^2}{\partial z_2 \partial \bar{z}_2} F(z,0) = \dfrac{2\lambda(1+2\lambda|z|^2)}{3(1-|z|^2)^{2\lambda}}\\
\end{align*}
Substituting the values of 
\begin{align*}
F(z,0), \frac{\partial}{\partial z_i} F (z, 0), \frac{\partial}{\partial \bar{z}_j} F (z, 0), \frac{\partial^2}{\partial z_i \partial \bar{z}_j} F(z,0)\,\,\mbox{and}\,\, \frac{\partial^2}{\partial z_i \partial \bar{z}_j}\log G((z,0),(z,0))    
\end{align*}
in the expression of $\mathcal{K}_{i \bar{j}}^{(\lambda, \mu,\nu)}((z, 0))$,
we obtain the expression of $\mathcal{K}_{i \bar{j}}^{(\lambda, \mu,\nu)}((z, 0))$ as stated in the statement of this proposition.
\end{proof}

\begin{theorem}\label{thm:symeqv}
For $i = 1, 2$, consider positive real numbers $\mu_i, \nu_i, \eta_i$. Then, the $\mob$-homogeneous analytic Hilbert modules $\mathcal{A}^{(\lambda_1,\mu_1, \nu_1,)}$ and  $\mathcal{A}^{(\lambda_2,\mu_2, \nu_2)}$ are unitarily equivalent if and only if $\lambda_1=\lambda_2,\,\,\mu_1=\mu_2$ and $\nu_1=\nu_2$.
\end{theorem}

\begin{proof}
Let us assume that the analytic Hilbert modules $\mathcal{A}^{(\lambda_1,\mu_1, \nu_1)}$ and  $\mathcal{A}^{(\lambda_2\mu_2, \nu_2)}$ are unitarily equivalent. From Proposition \ref{inequiv:prop4}, it follows that the curvature matrix $\mathbb{K}^{(\lambda_i\mu_i, \nu_i)}(0,0)$ of $\mathcal{A}^{(\lambda_i,\mu_i, \nu_i)}$ is 
\[\mathbb{K}^{(\lambda_i,\mu_i, \nu_i)}(0,0) =
\begin{pmatrix}
\mu_i+\dfrac{3\lambda_i}{2} &
\dfrac{3\lambda_i}{2} \\[8pt]
\dfrac{3\lambda_i}{2} &
\nu_i+\dfrac{3\lambda_i}{2}
\end{pmatrix}.\]
Now, equating $\mathbb{K}^{(\lambda_1,\mu_1, \nu_1)}(0,0)$ and $\mathbb{K}^{(\lambda_2,\mu_2, \nu_2)}(0,0)$, we obtain $\mu_1=\mu_2,\,\nu_1=\nu_2,$ and $\lambda_1=\lambda_2$.
\end{proof}

The proof of the following theorem is similar to the proof of Theorem \ref{thm:symeqv} and therefore, the proof is excluded. 

\begin{theorem}\label{thm:traneqv}
For $i = 1, 2$, let $\alpha_i, \beta_i, \gamma_i$ be positive real numbers such that $\alpha_i,\beta_i>\gamma_i$. The analytic Hilbert modules  $\mathcal{H}^{(\alpha_1, \beta_1, \gamma_1)}$ and  $\mathcal{H}^{(\alpha_2, \beta_2, \gamma_2)}$ are unitarily equivalent if and only if $\alpha_1=\alpha_2,\,\beta_1=\beta_2$ and $\gamma_1=\gamma_2$.  
\end{theorem}

For positive real numbers $\alpha, \beta, \gamma, \eta$ satisfying $\alpha, \beta > \gamma$, let $\mathbb{B}^{(\alpha, \beta, \gamma, \eta)}$ be the curvature matrix of the kernel $\det \mathcal{B}^{(\alpha, \beta, \gamma, \eta)}$. In the following proposition, we provide a complete description of the curvature matrix $\mathbb{B}^{(\alpha, \beta, \gamma, \eta)}$.

\begin{prop}\label{curv det kernel}
Let $\mathbb{B}^{(\alpha,\beta,\gamma,\eta)}_{i\bar{j}}$, $1 \leq i, j\leq 2$, denote the $(i, j)$th entry of the curvature matrix $\mathbb{B}^{(\alpha,\beta,\gamma,\eta)}$ of the analytic Hilbert module $\mathbb{H}^{(\alpha,\beta,\gamma,\eta)}$. Then, for any $z \in \mathbb{D}$, we have
\begin{align*}
&\mathbb{B}^{(\alpha,\beta,\gamma,\eta)}_{1\bar{1}}(z,0)=\frac{1}{(1-|z|^2)^2}\left[ 2\alpha\eta+\frac{2\alpha\beta \big[\alpha\beta-\gamma^2 (1-|z|^2)^2 (1+2|z|^2)\big]}{\big(\alpha\beta-\gamma^2 (1-|z|^2)^2\big)^2}\right],\\
&\mathbb{B}^{(\alpha,\beta,\gamma,\eta)}_{1\bar{2}}(z,0) = \mathbb{B}^{(\alpha,\beta,\gamma,\eta)}_{2\bar{1}}(z,0)= -2\gamma\eta-\frac{2\gamma^2 (1-|z|^2)\big[\alpha\beta(1-3|z|^2)-\gamma^2 (1-|z|^2)^3\big]}{\big(\alpha\beta-\gamma^2 (1-|z|^2)^2\big)^2}\,\,\mbox{and}\\
&\mathbb{B}^{(\alpha,\beta,\gamma,\eta)}_{2\bar{2}}(z,0)= 2\beta\eta+\frac{2\alpha\beta \big[\alpha\beta-\gamma^2 (1-|z|^2)^2 (1+2|z|^2)\big]}{\big(\alpha\beta-\gamma^2 (1-|z|^2)^2\big)^2}.
\end{align*}
\end{prop}

\begin{proof}
Suppose $F(\bl z) = \frac{\alpha\beta}{(1-|z_1|^2)^2(1-|z_2|^2)^2}-\frac{\gamma^2}{|1-z_1\bar z_2|^4}$ for $\bl z=(z_1,z_2)\in \mathbb{D}^2$. Then, we have $\det \mathcal{ B }^{ (\alpha,\beta,\gamma,\eta) }(\bl z, \bl z)=  [B^{ (\alpha,\beta,\gamma) }(\bl z,\bl z)]^{2\eta}  F(\bl z)$, $\bl z \in \mathbb{D}^2$. 
    For $1 \leq i, j \leq 2$ and $z \in \mathbb{D}$, note that
\begin{align*}
\nonumber \mathbb{B}^{(\alpha,\beta,\gamma,\eta)}_{i \bar{j}}(z, 0) &= \frac{\partial^2}{\partial z_i \partial \bar{z}_j} \log \det \mathcal{ B }^{ (\alpha,\beta,\gamma,\lambda) }((z, 0), (z, 0))\\
\end{align*}
\begin{align}\label{eqn:det-tran2}
\nonumber &= 2\eta \frac{\partial^2}{\partial z_i \partial \bar{z}_j} \log B^{(\alpha, \beta, \gamma)}((z, 0), (z, 0)) + \frac{\partial^2}{\partial z_i \partial \bar{z}_j} \log F(z, 0)\\
\nonumber &= 2\eta\mathcal{B}^{(\alpha,\beta,\gamma)}_{i \bar{j}}(z, 0)+\frac{\partial^2}{\partial z_i \partial \bar{z}_j} \log F(z, 0)\\
 &= 2\eta\mathcal{B}^{(\alpha,\beta,\gamma)}_{i \bar{j}}(z, 0) + \dfrac{F(z, 0) \frac{\partial^2}{\partial z_i \partial \bar{z}_j} F(z, 0)-\frac{\partial}{\partial z_i}F(z, 0) \frac{\partial}{\partial \bar{z}_j} F(z, 0)}{F(z, 0)^2}.
\end{align}
Replacing $\bl z$ by $(z, 0),$ $z \in \mathbb{D},$ in the expression of $F$, we obtain 
$$F(z,0)=\dfrac{\alpha\beta}{(1-|z|^2)^2}-\gamma^2.$$
For an arbitrary $z \in \mathbb{D}$, a direct computation shows that
\begin{align*}
&\frac{\partial}{\partial z_1} F (z, 0) =\frac{2\alpha\beta\bar{z}}{(1-|z|^2)^3} \\
&\frac{\partial}{\partial z_2} F(z, 0) = -2\gamma^2\bar{z}\\
&\frac{\partial}{\partial \bar{z}_1} F (z, 0)= \frac{2\alpha\beta z}{(1-|z|^2)^3}\\
& \frac{\partial}{\partial \bar{z}_2} F(z, 0) =-2\gamma^2 z \\
&\frac{\partial^2}{\partial z_1 \partial \bar{z}_1} F(z, 0)=\dfrac{(2+4|z|^2)\alpha\beta}{(1-|z|^2)^2} \\
&\frac{\partial^2}{\partial z_1 \partial \bar{z}_2} F(z, 0)= -2\gamma^2\\
&\frac{\partial^2}{\partial z_2 \partial \bar{z}_1} F(z, 0)= -2\gamma^2\\
&\frac{\partial^2}{\partial z_2 \partial \bar{z}_2} F(z, 0)= \dfrac{2\alpha\beta}{(1-|z|^2)^2}-4\gamma^2|z|^2.\\
\end{align*}
Substituting the values of 
\begin{align*}
 F(z, 0), \frac{\partial}{\partial z_i} F (z, 0), \frac{\partial}{\partial \bar{z}_j} F(z, 0),\frac{\partial^2}{\partial z_i \partial \bar{z}_j} F (z, 0) \,\,\mbox{and} \,\,\ \mathcal{B}^{(\alpha,\beta,\gamma)}_{i \bar{j}}(z,0)
\end{align*}
in Equation \eqref{eqn:det-tran2}, we obtain the expression of $\mathbb{B}^{(\alpha,\beta,\gamma,\lambda)}_{i\bar{j}}((z, 0))$ as stated in the statement of this proposition.
\end{proof}

The following theorem shows that the unitary equivalence class of the module $\mathbb{H}^{(\alpha, \beta, \gamma, \eta)}$ depends on $\frac{\alpha}{\beta}$, $\frac{\beta}{\gamma}$ and $\gamma \eta$.

\begin{theorem}\label{inequiv of det kernel}
For $i = 1, 2$, suppose $\alpha_i, \beta_i, \gamma_i, \eta_i$ are positive real numbers such that $\alpha_i,\beta_i>\gamma_i$. Then, the analytic Hilbert modules  $\mathbb{H}^{(\alpha_1, \beta_1, \gamma_1, \eta_1)}$ and  $\mathbb{H}^{(\alpha_2, \beta_2, \gamma_2, \eta_2)}$ are unitarily equivalent if and only if $\frac{\alpha_1}{\alpha_2}=\frac{\beta_1}{\beta_2}=\frac{\gamma_1}{\gamma_2}=\frac{\eta_2}{\eta_1}$.  
\end{theorem}

\begin{proof}
Assume that the analytic Hilbert modules $\mathbb{H}^{(\alpha_1, \beta_1, \gamma_1, \eta_1)}$ and  $\mathbb{H}^{(\alpha_2, \beta_2, \gamma_2, \eta_2)}$ are equivalent. 
Recall that the curvature matrix of the analytic Hilbert module $\mathbb{H}^{(\alpha_i, \beta_i, \gamma_i, \eta_i)}$ is denoted by $\mathbb{B}^{(\alpha_i, \beta_i, \gamma_i, \eta_i)}$. Due to our assumption, we have 
\begin{equation}\label{eqn:detcurvineqiv2}
\mathbb{B}^{(\alpha_1, \beta_1, \gamma_1, \eta_1)}(z, 0) = \mathbb{B}^{(\alpha_2, \beta_2, \gamma_2, \eta_2)} (z, 0),\,\, z \in \mathbb{D}.    
\end{equation}
In particular, for $i, j = 1 , 2$, we have $\mathbb{B}_{i \bar{j}}^{(\alpha_1, \beta_1, \gamma_1, \eta_1)}(z, 0) = \mathbb{B}_{i \bar{j}}^{(\alpha_2, \beta_2, \gamma_2, \eta_2)}(z, 0)$, $z \in \mathbb{D}$. Note that
$$\mathbb{B}^{(\alpha_i,\beta_i,\gamma_i,\eta_i)}_{1\bar{1}}(z,0)=\frac{1}{(1-|z|^2)^2}\left[ 2\alpha_i\eta_i + \frac{2\alpha_i \beta_i \left[\alpha_i \beta_i -\gamma_i^2 (1-|z|^2)^2 (1+2|z|^2)\right]}{\left(\alpha_i\beta_i-\gamma_i^2 (1-|z|^2)^2\right)^2}\right],\,\,z \in \mathbb{D},$$
thanks to Proposition \ref{curv det kernel}. Equation \eqref{eqn:detcurvineqiv2} implies that
$$(1 - |z|^2)^2 \mathbb{B}_{1 \bar{1}}^{(\alpha_1, \beta_1, \gamma_1, \eta_1)}(z, 0) = (1 - |z|^2)^2\mathbb{B}_{1 \bar{1}}^{(\alpha_2, \beta_2, \gamma_2, \eta_1)}(z, 0)$$
holds for every $z \in \mathbb{D}$. Now, taking $|z| \to 1$ in the above equation, we obtain
\begin{equation}\label{eqn:crvdetinequiv1}
 2\alpha_1 \eta_1 + 2 = 2 \alpha_2 \eta_2 + 2.   
\end{equation}

Similarly, equating $\mathbb{B}^{(\alpha_1,\beta_1,\gamma_1,\eta_1)}_{2\bar{2}}(z,0)$, $\mathbb{B}^{(\alpha_2,\beta_2,\gamma_2,\eta_2)}_{2\bar{2}}(z,0)$ and then considering $|z| \to 1$, we obtain
\begin{equation}\label{eqn:crvdetinequiv3}
 2\beta_1 \eta_1 + 2 = 2 \beta_2 \eta_2 + 2.   
\end{equation}
Finally, equating $\mathbb{B}^{(\alpha_1,\beta_1,\gamma_1,\eta_1)}_{1\bar{2}}(z,0)$, $\mathbb{B}^{(\alpha_2,\beta_2,\gamma_2,\eta_2)}_{1\bar{2}}(z,0)$ and then taking $|z| \to 1$, we obtain
\begin{equation}\label{eqn:crvdetinequiv4}
 2\gamma_1 \eta_1 = 2 \gamma_2 \eta_2 .   
\end{equation}
Equations \eqref{eqn:crvdetinequiv1}, \eqref{eqn:crvdetinequiv3}, \eqref{eqn:crvdetinequiv4} together imply that $\frac{\alpha_1}{\alpha_2}=\frac{\beta_1}{\beta_2}=\frac{\gamma_1}{\gamma_2}=\frac{\eta_2}{\eta_1}$.  \\

 Conversely, assume that $\frac{\alpha_1}{\alpha_2}=\frac{\beta_1}{\beta_2}=\frac{\gamma_1}{\gamma_2}=\frac{\eta_2}{\eta_1}$. Hence, 
\begin{align*}
\alpha_1 \eta_1=\alpha_2 \eta_2,\beta_1 \eta_1=\beta_2 \eta_2,\gamma_1 \eta_1=\gamma_2 \eta_2\,\mbox{and} \,\frac{\alpha_1\beta_1}{\gamma_1^2}=\frac{\alpha_2\beta_2}{\gamma_2^2}
\end{align*}
\begin{align*}
    \mathbb{B}^{(\alpha_1,\beta_1,\gamma_1,\eta_1)}_{1\bar{1}}(z,0)&=\frac{1}{(1-|z|^2)^2}\left[ 2\alpha_1\eta_1 + \frac{2\alpha_1 \beta_1 \left[\alpha_1 \beta_1 -\gamma_1^2 (1-|z|^2)^2 (1+2|z|^2)\right]}{\left(\alpha_1\beta_1-\gamma_1^2 (1-|z|^2)^2\right)^2}\right]\\
    &=\frac{1}{(1-|z|^2)^2}\left[ 2\alpha_1\eta_1 + \frac{2\alpha_1 \beta_1 \left[\frac{\alpha_1\beta_1}{\gamma_1^2} - (1-|z|^2)^2 (1+2|z|^2)\right]}{\gamma_1^2\left(\frac{\alpha_1\beta_1}{\gamma_1^2}- (1-|z|^2)^2\right)^2}\right]\\
    &=\frac{1}{(1-|z|^2)^2}\left[ 2\alpha_2\eta_2 + \frac{2\alpha_2 \beta_2 \left[\frac{\alpha_2\beta_2}{\gamma_2^2} - (1-|z|^2)^2 (1+2|z|^2)\right]}{\gamma_2^2\left(\frac{\alpha_2\beta_2}{\gamma_2^2}- (1-|z|^2)^2\right)^2}\right]\\
    &=\frac{1}{(1-|z|^2)^2}\left[ 2\alpha_2\eta_2 + \frac{2\alpha_2 \beta_2 \left[\alpha_2 \beta_2 -\gamma_2^2 (1-|z|^2)^2 (1+2|z|^2)\right]}{\left(\alpha_2\beta_2-\gamma_2^2 (1-|z|^2)^2\right)^2}\right]\\
    &=\mathbb{B}^{(\alpha_2,\beta_2,\gamma_2,\eta_2)}_{2\bar{2}}(z,0)
\end{align*}

Similarly, using the equations $\gamma_1\eta_1=\gamma_2 \eta_2\,\mbox{and} \,\frac{\alpha_1\beta_1}{\gamma_1^2}=\frac{\alpha_2\beta_2}{\gamma_2^2}$, we obtain 
$$\mathbb{B}^{(\alpha_1,\beta_1,\gamma_1,\eta_1)}_{1\bar{2}}(z,0)=\mathbb{B}^{(\alpha_2,\beta_2,\gamma_2,\eta_2)}_{1\bar{2}}(z,0)\,\, \mbox{and}\,\, \mathbb{B}^{(\alpha_1,\beta_1,\gamma_1,\eta_1)}_{2\bar{1}}(z,0)=\mathbb{B}^{(\alpha_2,\beta_2,\gamma_2,\eta_2)}_{2\bar{1}}(z,0).$$

Finally, the equations $\beta_1\eta_1=\beta_2 \eta_2\,\mbox{and} \,\frac{\alpha_1\beta_1}{\gamma_1^2}=\frac{\alpha_2\beta_2}{\gamma_2^2}$ gives us 
$$\mathbb{B}^{(\alpha_1,\beta_1,\gamma_1,\lambda_1)}_{2\bar{2}}(z,0)=\mathbb{B}^{(\alpha_2,\beta_2,\gamma_2,\lambda_2)}_{2\bar{2}}(z,0).$$
This proves that $\mathbb{B}^{(\alpha_1, \beta_1, \gamma_1, \eta_1)}(z,0) = \mathbb{B}^{(\alpha_2, \beta_2, \gamma_2, \eta_2)}(z, 0)$ holds for every $z \in \mathbb{D}$. 
Since the analytic Hilbert modules $\mathbb{H}^{(\alpha_1, \beta_1, \gamma_1, \eta_1)}$ and  $\mathbb{H}^{(\alpha_2, \beta_2, \gamma_2, \eta_2)}$ are $\mob$-homogeneous, it follows from Proposition \ref{prop:curv} that $\mathbb{B}^{(\alpha_1, \beta_1, \gamma_1, \eta_1)}(\boldsymbol{z}) = \mathbb{B}^{(\alpha_2, \beta_2, \gamma_2, \eta_2)}(\boldsymbol{z})$ holds for every $\boldsymbol{z} \in \mathbb{D}^2$. Therefore, the Hilbert modules  $\mathbb{H}^{(\alpha_1, \beta_1, \gamma_1, \eta_1)}$ and  $\mathbb{H}^{(\alpha_2, \beta_2, \gamma_2, \eta_2)}$ are unitary equivalent.
\end{proof}

The following theorem shows that three families of M\"ob-homogeneous analytic Hilbert modules, described in Section \ref{sec 3}, are mutually distinct up to unitary equivalence and each module in these three families is unitarily inequivalent to a weighted Bergman module.  

\begin{theorem}\label{inequivalance}
    No M\"ob-homogeneous analytic Hilbert module belonging to any one of the three families $\{\mathcal{A}^{(\lambda,\mu,\nu)} : \lambda, \mu, \nu > 0\}$, $\{\mathcal{H}^{(\alpha, \beta, \gamma)} : \alpha, \beta > \gamma > 0\}$ and $\{\mathbb{H}^{(\alpha, \beta, \gamma, \eta)} : \alpha, \beta > \gamma > 0, \eta \geq 0\}$ is unitarily equivalent to a module belonging to another family. Moreover, no Hilbert module belonging to any of the above three families is unitarily equivalent to a weighted Bergman module over the bidisc.
\end{theorem}

\begin{proof}
Suppose $\lambda,\mu, \nu, \eta, \alpha, \beta, \gamma$ are positive real numbers such that $\alpha, \beta > \gamma$. We first prove that the analytic Hilbert modules  $\mathcal{A}^{(\lambda,\mu,\nu)}$ and  $\mathbb{H}^{(\alpha, \beta, \gamma, \eta)}$ are not unitarily equivalent. 

Assume that $\mathcal{A}^{(\lambda,\mu,\nu)}$ and  $\mathbb{H}^{(\alpha, \beta, \gamma, \eta)}$ are unitarily equivalent. Recall that $\mathbb{K}^{(\lambda,\mu,\nu)}(\bl z)$ and $\mathbb{B}^{(\alpha, \beta, \gamma, \eta)}(\bl z)$ are curvature matrices of $\mathcal{A}^{(\lambda,\mu,\nu)}$ and  $\mathbb{H}^{(\alpha, \beta, \gamma, \eta)}$, respectively at $\bl z \in \mathbb{D}^2$. Equating $\mathcal{K}^{(\lambda,\mu,\nu)}_{1\bar{1}}(0,0)$  and $\mathbb{B}^{(\alpha, \beta, \gamma, \eta)}_{1\bar{1}}(0,0)$, we have
\begin{equation}\label{eqn:detsyminequiv1}
\mu+\frac{3\lambda}{2}=2\alpha \left[ \eta  + \frac{\beta}{\alpha\beta - \gamma^2} \right].   
\end{equation}
Finally, equating $\mathcal{K}^{(\lambda,\mu, \nu)}(z, 0)$ and $\mathbb{B}^{(\alpha, \beta, \gamma, \eta)}(z, 0)$ for $z \in \mathbb{D}$ and then, considering $|z| \to 1$, we obtain
\begin{equation}\label{eqn:detsyminequiv2}
\mu+3\lambda=2\alpha\eta + 2 .   
\end{equation}
From Equation \eqref{eqn:detsyminequiv1} and \eqref{eqn:detsyminequiv2}, we get
$\lambda=\frac{2}{3}\left[\frac{-2\gamma^2}{\alpha\beta-\gamma^2}\right]$ which is a negative real number since $\alpha, \beta > \gamma$. This contradicts that $\lambda > 0$.

Let $\alpha', \beta'$ and $\gamma'$ be positive real numbers such that $\alpha', \beta' > \gamma'$. A similar proof shows that the Hilbert modules $\mathcal{A}^{(\lambda,\mu, \nu)}$ and $\mathcal{H}^{(\alpha, \beta, \gamma)}$ as well as $\mathbb{H}^{(\alpha, \beta, \gamma, \eta)}$ and $\mathcal{H}^{(\alpha' \beta',\gamma')}$ are not unitarily equivalent.

Suppose $\rho_1$ and $\rho_2$ are two positive real numbers. We prove that the Hilbert modules $\mathcal{H}^{(\alpha, \beta, \gamma)}$ and the weighted Bergman module $\mathbb{A}^{(\rho_1, \rho_2)}(\mathbb{D}^2)$ are not unitarily equivalent. 

Let $\Theta_{i\bar j}^{(\rho_1, \rho_2)}$, $1 \le i,j \le 2$, denote the $(i,j)$ th entry of the curvature matrix $\Theta^{(\rho_1, \rho_2)}$ of $\mathbb{A}^{(\rho_1, \rho_2)}$. A direct computation shows that
\begin{align*}
    \Theta_{1\bar 1}^{(\lambda,\mu)}(z,0)=\frac{\rho_1}{(1-|z|^2)^2},\, \Theta_{1\bar 2}^{(\lambda,\mu)}(z,0)= \Theta_{2\bar 1}^{(\lambda,\mu)}(z,0)=0\,\, \mbox{and} \,\, \Theta_{2\bar 2}^{(\lambda,\mu)}(z,0)=\rho_2,
    \end{align*}
for any $z \in \mathbb{D}$. Let us assume that, $\mathcal{H}^{(\alpha, \beta, \gamma)}$ and $\mathbb{A}^{(\rho_1, \rho_2)}$ are unitarily equivalent. Equating,  $\mathcal{B}_{1 \bar2}^{(\alpha, \beta, \gamma, \lambda)}(0,0)$ and $\Theta_{1\bar 2}^{(\rho_1,\rho_2)}(0,0)$, we have $\gamma=0$. This contradicts that $\gamma > 0$.

A similar proof shows that the Hilbert modules $\mathcal{A}^{(\lambda,\mu, \nu)}$ and $\mathbb{A}^{(\rho_1, \rho_2)}$ as well as $\mathbb{H}^{(\alpha, \beta, \gamma, \eta)}$ and $\mathbb{A}^{(\rho_1, \rho_2)}$ are not unitarily equivalent.
\end{proof}

\begin{remark} \label{lb}
Let $G_{(\alpha,\beta,\gamma,\eta)}$ be the  trivial holomorphic line bundle over $\mathbb{D}^2$ equipped with the Hermitian metric $\det \mathcal B^{(\alpha,\beta,\gamma,\eta)} (\bl{z}, \bl{z})$, $\bl{z} \in \mathbb D^2$  and $F_{(\alpha',\beta',\gamma')}$ be the trivial holomorphic line bundle over $\mathbb{D}^2$ equipped with the Hermitian metric  $B_{(\alpha',\beta',\gamma')}(\boldsymbol{z}, \boldsymbol{z})$, $\boldsymbol{z} \in \mathbb D^2$. From Theorem \ref{inequivalance}, it follows that $G_{(\alpha,\beta,\gamma,\eta)}$ is not equivalent to $F_{(\alpha',\beta',\gamma')}$ for any positive reals $\alpha,\beta,\gamma,\eta,\alpha',\beta',\gamma'$ satisfying $ \alpha,\beta>\gamma$ and $\alpha',\beta'>\gamma'$.    
\end{remark}

Let  $ ( M, \omega ) $ be an $ n $-dimensional K\"ahler manifold with the K\"ahler form $ \omega $ given by $ \omega = \displaystyle \sum_{ i, j = 1 }^n g_{ i \bar{ j } } dz_i \wedge d\bar{z}_j $.  Recall that the Ricci form $ \mathrm{Ric} $ of $ ( M, \omega ) $ is defined as 
    $$ \mathrm{Ric} : = - \sqrt{ - 1 } \sum_{ i, j = 1 }^n \partial_i \bar{ \partial_j } \log \det \big( \! \!\big( g_{ i \bar{ j } } \big) \!\! \big)_{ i, j = 1 }^n dz_i \wedge d\bar{z}_j  $$
 A K\"ahler manifold is said to be K\"ahler-Einstein if the Ricci form $ \mathrm{Ric} $ is a scalar multiple of the  K\"ahler form $ \omega $.

\begin{cor}
    Suppose $\alpha, \beta, \gamma$ are positive real numbers such that $\alpha, \beta > \gamma$. The metric $ \displaystyle \sum_{ i, j = 1 }^2 \mathcal{B}_{ i \bar{ j } }^{ ( \alpha.\beta,\gamma ) } dz_i \wedge d \bar{ z_j } $ on $ \mathbb{ D }^2 $ is not K\"ahler-Einstein.
\end{cor}

\begin{proof}
Assume that there exists $\upsilon \in \mathbb{R}$ such that the Ricci form $ \mathrm{Ric}^{ (\alpha,\beta,\gamma ) } $ on $ \mathbb{ D }^2 $ of the metric $ \displaystyle \sum_{ i, j = 1 }^2 \mathcal{B}_{ i \bar{ j } }^{ ( \alpha.\beta,\gamma ) } dz_i \wedge d \bar{ z_j } $ on $\mathbb{D}^2$ satisfies
$$ \mathrm{Ric}^{ (\alpha,\beta,\gamma ) } = \upsilon \sum_{ i, j = 1 }^2 \mathcal{B}_{ i \bar{ j } }^{ (\alpha,\beta,\gamma ) } dz_i \wedge d \bar{ z }_j . $$
Then, the definition of $ \mathrm{ Ric } $ implies that
$$ \sum_{ i, j = 1 }^2 \partial_i \bar{ \partial_j } \log \det \big( \! \!\big( \mathcal{B}_{ i \bar{ j } }^{ (\alpha,\beta,\gamma ) } \big) \!\! \big)_{ i, j = 1 }^2 dz_i \wedge d\bar{z_j}  = \upsilon \sum_{ i, j = 1 }^2 \mathcal{B}_{ i \bar{ j } }^{ (\alpha,\beta,\gamma ) } dz_i \wedge d \bar{ z_j } . $$
Replacing $ \mathcal{B}_{ i \bar{ j } }^{ (\alpha,\beta,\gamma ) }  = \partial_i \bar{ \partial_j } \log B_{ ( \alpha,\beta,\gamma ) } $, $ 1 \leq i, j \leq 2 $, in the equation above, we obtain 
\begin{equation}\label{kheq}
\partial_i \bar{ \partial_j } \log \det \big[ \big(B_{ ( \alpha,\beta,\gamma )} \big)^{ \eta } \big( \! \!\big( \mathcal{B}_{ i \bar{ j } }^{ (\alpha,\beta,\gamma )  } \big) \!\! \big)_{ i, j = 1 }^2 \big] =   \partial_i \bar{ \partial_j } \log \big(  B_{ ( \alpha,\beta,\gamma )  } \big)^{ 2 \eta + \upsilon } , ~ 1 \leq i, j \leq 2 .    \end{equation} 

From Remark \ref{lb}, we know that the trivial Hermitian holomorphic line bundles $G_{(\alpha,\beta,\gamma,\eta)}$ and $F_{(\alpha',\beta',\gamma')}$ over $\mathbb{D}^2$ are inequivalent for any positive real numbers $\alpha, \alpha',\beta,\beta',\gamma,\gamma',\eta$.
 Therefore, choosing $\eta > 0$ such that $2\eta + \upsilon > 0$, we obtain a contradiction from Equation \eqref{kheq}.
\end{proof}

\textit{Acknowledgement}.
The authors are indebted to Professor Gadadhar Misra and Dr. Dinesh Kumar Keshari for the proof of Proposition \ref{propB1}. The second named author is also grateful to Dr. Subrata Shyam Roy for pointing out the proof that $K^{(\lambda)}_{\text{\tiny{anti}}}$ is quasi-invariant for every $\lambda > 0$.

\bibliographystyle{amsplain}

\end{document}